\documentclass[12pt,a4paper]{article}
\usepackage{amsmath}    
\usepackage{amssymb}    
\usepackage{amsthm}     
\usepackage{graphics}   

\usepackage{hyperref}
\hypersetup{
breaklinks,
colorlinks=true,
linkcolor=blue,
pdfborder = 0 0 0,
citecolor=blue,
urlcolor=blue}

\def\fspace{E} 
\def\B{{\cal B}} 

\numberwithin{equation}{section}

\makeatletter\@twosidetrue\makeatother

\newtheorem{theorem}{Theorem}[section]
\newtheorem{lemma}[theorem]{Lemma}
\newtheorem{corollary}[theorem]{Corollary}
\newtheorem{proposition}[theorem]{Proposition}

\theoremstyle{definition}
\newtheorem{definition}[theorem]{Definition}
\newtheorem{example}[theorem]{Example}
\newtheorem{remark}[theorem]{Remark}

\newcommand{\R}{\mathbb{R}}

\title{\Large\textbf{Symmetry and monotonicity results for solutions of semilinear PDEs in sector-like domains}}
\author{Francesca Gladiali and Antonio Greco
}

\date{}

\begin{document}

\definecolor{antonio}{rgb}{.0,.5,.0}
\belowdisplayshortskip = \belowdisplayskip

\paperwidth=210 true mm
\paperheight=297 true mm
\pdfpagewidth=210 true mm
\pdfpageheight=297 true mm

\maketitle

\pdfbookmark[2]{Abstract}{Abstract}
\begin{abstract}
In this paper we consider semilinear PDEs, with a convex nonlinearity,
in a sector-like domain. Using cylindrical coordinates $(r, \theta, z)$,  we
investigate the shape of solutions whose derivative in $\theta$ vanishes at
the boundary. We prove that any solution with Morse index less than
 two must be either independent of~$\theta$ or strictly monotone with respect to~$\theta$. In the special case of a planar domain, the result  holds in a circular sector as well as in an annular, and it can also be extended to a rectangular domain. The corresponding problem in higher dimensions is also considered, as well as an extension to unbounded domains. The proof is based on a rotating-plane argument: a convenient manifold is introduced in order to avoid overlapping the domain with its reflected image in the case when its opening is larger than~$\pi$.
\end{abstract}

\section{Introduction}

In this paper we consider solutions of semilinear elliptic PDEs 
in a sector-like domain $\Omega_{0\beta} \subset \mathbb R^N$, $N \ge 2$. To describe the results, for every $x = (x_1, \ldots, x_N)\in \mathbb R^N$ we define $d(x) = \sqrt{x_1^2 + x_2^2 \,}$ and $z = (x_3, \ldots, x_N) \in \mathbb R^{N - 2}$, and we denote by $\Upsilon$ the $(N - 2)$-dimensional subspace $\Upsilon :=\{x\in \R^N: d(x)=0\}$.
Observe that each point $x\in \R^N\setminus\Upsilon$ can be represented in cylindrical coordinates as $i(x) = (r,\theta,z)$, where $r = d(x)$, and $\theta = \theta(x) \in [0,2\pi)$ is uniquely determined by
$$
\begin{cases}
x_1 = r \, \cos \theta,
\\
x_2 = r \, \sin \theta.
\end{cases}
$$
If $N = 2$, then $(r,\theta)$ are the polar coordinates of~$x$, the set $\Upsilon$ contains just the origin, and the symbol~$z$ should be ignored. In general, a rotation of~$\mathbb R^N$ about the origin is a linear mapping $x \mapsto Ax$ associated to an orthogonal matrix~$A$. In the present paper we deal with \textit{cylindrical} rotations, i.e., the particular rotations whose matrix $A = A^N_\theta$ is given by
$$
A^2_\theta
=
\begin{pmatrix}
\cos\theta & \sin\theta
\\
-\sin\theta & \cos\theta
\end{pmatrix}
\quad
\mbox{and}
\quad
A^N_\theta
=
\begin{pmatrix}
A^2_\theta & 0\\
0 & I^{N - 2}\\
\end{pmatrix}
\!,\
N \ge 3,
$$
where $I^{N -2}$ is the $(N - 2)$-dimensional identity matrix. In the first part of the paper we deal with bounded domains. More precisely, let $\Omega$ be a bounded, Lipschitz domain which is invariant under cylindrical rotations in the sense that $x \in \Omega$ if and only if $A_\theta \, x \in \Omega$ for every $\theta \in [0,2\pi)$. For $\theta_1,\theta_2\in [0,2\pi)$ satisfying $\theta_1<\theta_2$ we introduce the bounded open sector 
\[\Omega_{\theta_1\theta_2}:=\{\, x \in \Omega\setminus \Upsilon : \theta(x) \in (\theta_1,\theta_2) \,\}\]
whose boundary is made of the open, flat surfaces $\Gamma_{\!\theta_i}:=\{\, x \in \Omega\setminus \Upsilon : \theta(x) = \theta_i \,\}$, $i=1,2$, the torical surface $\gamma_{\theta_1\theta_2}:= \{\, x \in \partial\Omega\setminus \Upsilon : \theta(x) \in [\theta_1,\theta_2] \,\}$ and the (possibly empty) set $\gamma:=\overline \Omega\cap \Upsilon$. The typical examples of the domain $\Omega_{\theta_1\theta_2}$ are: a sector of a sphere, a sector of a cylinder, a sector of an annulus, a sector of a torus, a sector of a cone. But also more complicated domains as cylinders (or spheres) with cavities can be considered for~$\Omega$, and even tori with a torical, coaxial tunnel inside. Observe that $\gamma$ may be empty (which is the case, for instance, when $\Omega$ is a torus), and can be disconnected as in the case when $\Omega$ is a cylinder in~$\mathbb R^3$ with a cavity.
Now fix $\beta\in (0,2\pi)$ and consider the boundary-value problem
\begin{equation}\label{P}
\left\{\begin{array}{ll}
-\Delta u = f(d(x),z,u) \qquad & \text{ in } \Omega_{0\beta}, \\
\noalign{\medskip}
\frac{\partial u}{\partial \theta}=0 & \text{ on $\gamma_{0\beta}\cup \Gamma_{\! 0} \cup \Gamma_{\! \beta}$,}
\end{array} \right.
\end{equation} 
where $\Delta$ is the Laplace operator, and $f(r,z,u)$ is locally H\"older continuous on $[0,+\infty) \times \mathbb R^{N - 1}$ {together with $f' = \partial f / \partial u$}. 
Observe that $u_\theta = \frac {\partial u}{\partial \theta}$ is the normal derivative of $u$ on $\Gamma_{\! 0} \cup \Gamma_{\! \beta}$ while it is the tangential derivative on $\gamma_{0\beta}$.
Let $\cal X$ be the set of all functions $u \in C^2(\Omega_{0\beta}) \cap C^0(\overline \Omega_{0\beta})\cap H^1(\Omega_{0\beta})$ such that $u_\theta \in H^1(\Omega_{0\beta})$. We investigate functions $u \in \cal X$ satisfying~\eqref{P} pointwise. Using \cite[Prop. 6.1]{PacellaTralli}, it can be proved that the solution $u$ is smooth, namely $C^2$, in $\Omega_{0\beta}\cup \Gamma_{\! 0} \cup \Gamma_{\! \beta}$.
The assumption that $u_\theta\in H^1(\Omega_{0\beta})$ is introduced to avoid technicalities and it is satisfied, for instance, if $u \in W^{2,2}(\Omega_{0\beta})$.
Such an assumption is dropped in the last part of the paper, where we use a truncation technique to deal with unbounded domains (see, in particular, \eqref{gradient}-\eqref{zeta}). The boundary condition $u_\theta = 0$ on $\gamma_{0\beta}$ can be interpreted in the sense of traces. However, since $u$ is continuous up to there, the condition implies that the values of~$u$ on~$\gamma_{0\beta}$ are constant in~$\theta$. Thus, \eqref{P} is a shorthand for the following mixed problem of Dirichlet-Neumann type:
\begin{equation}\label{instance}
\left\{\begin{array}{ll}
-\Delta u = f(d(x),z,u) \qquad & \text{ in } \Omega_{0\beta}, \\
\noalign{\medskip}
u=g(d(x),z) & \text{ on } \gamma_{0\beta},\\
\noalign{\medskip}
\frac{\partial u}{\partial \nu}=0 & \text{ on } \Gamma_{\! 0} \cup \Gamma_{\! \beta}
\end{array} \right.
\end{equation}
where $g(r,z)$ is a continuous function, and $\nu$ denotes the outward unit normal to $\partial \Omega_{0\beta}$. If $u$ is a solution to \eqref{P}, we denote by~$I$ the identity operator and by
\[L_u:=-\Delta-f'(d(x),z,u(x)) \, I\]
the linearized operator at~$u$. Furthermore,
\begin{equation}\label{Q}
Q_D(v):=\int_D \Big( |\nabla v(x)|^2-f'(d(x),z,u(x)) \, v^2(x) \Big) \, dx,
\quad
v \in H^1(D)
\end{equation}
is the quadratic form associated to $L_u$ in an open subset $D \subset \Omega_{0\beta}$. In fact, the operator~$L_u$ and the quadratic form~$Q_D$ are well defined for every $u \in L^\infty(\Omega_{0\beta})$. If $\Gamma$ is a sufficiently smooth (possibly disconnected) part of $\partial D$, we denote by
\[H^1_{\Gamma}(D):=\{v\in H^1(D): v=0 \text{ on }\Gamma\}\]
the Sobolev space of square-summable functions having square-summable weak derivatives and vanishing trace along~$\Gamma$, and, for $i\geq 1$, we denote by $\lambda_i(H^1_\Gamma(D))$ the eigenvalues of the operator $L_u$ in $H^1_\Gamma(D)$. We will assume that the solution $u$ to \eqref{P} satisfies $\lambda_2(H^1_{\gamma_{0\beta}}(\Omega_{0\beta}))\geq 0$.
Our main result in bounded domains is the following:

\begin{theorem}\label{teo-1}
Let $u \in \cal X$ be a solution to \eqref{P}, where $f(r,z,u)$ is convex with respect to $u$.
\begin{enumerate}
\item If\/ $\lambda_1(H^1_{\gamma_{0\beta}}(\Omega_{0\beta}))\geq 0$ then $u$ is constant with respect to the angular variable\/ $\theta$.
\item If\/ $\lambda_1(H^1_{\gamma_{0\beta}}(\Omega_{0\beta})) < 0 \le \lambda_2(H^1_{\gamma_{0\beta}}(\Omega_{0\beta}))$ then $u$ is either independent of\/~$\theta$ or strictly monotone with respect to~$\theta$ in\/ $\Omega_{0\beta}$. In the last case, $u$ has a nonvanishing derivative~$u_\theta$, the first Dirichlet eigenvalue $\lambda_1(H^1_0(\Omega_{0\beta}))$ equals zero and $u_\theta$ is a corresponding eigenfunction.
\end{enumerate}
\end{theorem}

\noindent In Section~\ref{cylindrical} we state a corresponding result for the case when $\Omega_{0\beta}$ is a cylindrical domain, which is in some sense the asymptotic shape of a sector-like domain when the opening~$\beta$ is small. All the occurrences in Theorem \ref{teo-1} can happen and we will provide some examples in Section~\ref{exe}.  In particular we show that if the first Dirichlet eigenvalue $\lambda_1(H^1_0(\Omega_{0\beta})) = 0$ then $u_\theta$ may well vanish identically. 
We also show that if $\lambda_2(H^1_{\gamma_{0\beta}}(\Omega_{0\beta}))< 0$ it can happen that the solution $u$ is neither independent of~$\theta$ nor monotone with respect to $\theta$.
Starting from Section~\ref{unbounded} we extend the result to unbounded domains: a precise statement is given in Theorem~\ref{teo-2}.

\begin{remark}
In the statement of Theorem \ref{teo-1} it is enough to require that $f(r,z,s)$ is convex with respect to~$s$ as long as $s \in (m,M)$, where $m=\inf\limits_{\Omega_{0\beta}}u(x)$ and~$M=\sup\limits_{\Omega_{0\beta}}u(x)$. 
This assumption has been introduced in the paper \cite{Pacella} in order to compare 
the quadratic form $Q_D(v)$ 
with the quadratic form associated to the equation satisfied by the difference of $u$ and its reflection with respect to some $\Gamma_{\!\theta}$ (as we also do in Section \ref{se:2}).\\
In the case of solutions to \eqref{instance} the assumption $\lambda_2(H^1_{\gamma_{0\beta}}(\Omega_{0\beta}))\geq 0$ is equivalent to consider solutions of Morse index $1$ and it satisfied by any ground state solution and any solution of mountain pass type. It is needed in the proof to start the ``rotating plane'' method in place of the positiveness of the solution $u$ which is usually required in the moving plane procedure. Moreover a bound on the Morse index usually implies an $L^\infty$-bound of the solution as observed in \cite{BahriLions} in the case of a Dirichlet problem under some growth condition on the nonlinear term $f$.\\
Some comments on the assumption $u_\theta\in H^1(\Omega_{0\beta})$, which is needed in the proof of Lemma~\ref{lem-3} and Lemma \ref{lem-4}, and consequently enters in Theorem \ref{teo-1}. 
By standard regularity theory, any weak solution $u \in H^1(\Omega_{0\beta})$ of~\eqref{instance} belongs to the smoothness class $C^2(\Omega_{0\beta}\cup \Gamma_{\! 0} \cup \Gamma_{\! \beta}) \cap C^0(\overline \Omega_{0\beta})$.
Furthermore, at least in the case when $\gamma=\emptyset$, if $\Omega\in C^{2,\alpha}$  and $g(d(x),z)\in C^{2,\alpha}$ then $u\in C^{2,\alpha}(\overline \Omega_{0\beta})$ showing that $u_\theta\in H^1(\Omega_{0\beta})$. \\
In the particular case of dimension $N=2$, the domain $\Omega_{0\beta}$ is a sector of the annulus or a sector of the disc so that $\gamma_{0\beta}$ is always smooth. In the case of the annulus then the assumption $u_\theta\in H^1(\Omega_{0\beta})$ is satisfied if $g(r)$ belongs to $C^{2,\alpha}$. In the case of the disc, instead, we can use \cite{Mazia} and \cite{AdolfssonJerison}
to gain the $W^{2,2}$ regularity of the solution $u$ when the domain is convex.
\\
Some previous results on cylindrical symmetry of low Morse index solutions
for a mixed problem can be found in \cite{Damascelli-Pacella-paper} where,  differently from our case, the equation holds in the whole of $\Omega$ and a nonlinear Neumann condition
is imposed on a subset of $\partial \Omega$.
Finally we recall that in the case when $\Omega_{0\beta}$ is a sector of the circle and $\beta<\pi$, the radial symmetry of the positive solution to \eqref{instance} with $g(d(x),z)=0$ was proved in~\cite{Berestycki-Pacella} using an involved version of the moving plane method.

\end{remark}

\section{Preliminary results}\label{se:2}

The proof of Theorem~\ref{teo-1} is based on the rotating plane method (rotating line, in two dimensions). The method was previously used in \cite{Pacella,PacellaWeth} to prove Schwarz symmetry for solutions of low Morse index to the Dirichlet problem in radially symmetric bounded domains. To be more precise, we use the reflection~$\sigma_\alpha$ with respect to~$\Gamma_{\!\alpha}$, where $\alpha$ is an angular parameter in the given interval $(0,\beta)$. In the special case when $\alpha \in (0,\pi)$ we may easily define $\sigma_\alpha \colon \Omega_{0,2\alpha} \to \Omega_{0,2\alpha}$ by prescribing that $i(\sigma_\alpha(x))=(r, \, 2\alpha - \theta, \, z)$ when $i(x)=(r,\theta,z)$. However, in order to simplify and generalize the subsequent arguments, it is convenient to extend the definition of~$\sigma_\alpha$ to the general case when $\alpha \in \R$. To this aim we consider the Riemannian manifold ${\cal M} :=\{(r,t,z): r \in (0,+\infty), t \in \mathbb R, z \in \mathbb R^{N-2}\}$ endowed with the flat metric whose first fundamental form is $ds^2 = dr^2 + r^2 \, dt^2 + dz^2$. Every sector-like domain $\Omega_{\theta_1\theta_2} \subset \R^N$ is isometrically embedded into~$\cal M$ by the mapping $i \colon \Omega_{\theta_1\theta_2} \to \cal M$ $i(x)=(r,\theta,z)$.
We will identify  $\Omega_{\theta_1\theta_2}$ with its image $i(\Omega_{\theta_1\theta_2})$ and, more generally, we will use the notation~$\Omega_{\theta_1\theta_2}$ to denote the submanifold
$$
\Omega_{\theta_1\theta_2}
=
\{\, (r,t,z) \in {\cal M}
:
t \in (\theta_1,\theta_2), (r,0,z)\in \Omega
\,\}
$$
for every $\theta_1,\theta_2 \in \R$ with $\theta_1 < \theta_2$. Now, for any given $\alpha \in \R$, we define the function $\sigma_\alpha \colon {\cal M} \to {\cal M}$ by letting $\sigma_\alpha(r,t,z) = (r, \, 2\alpha - t, \, z)$. Thus, for instance, if we take $\alpha \in (\pi,2\pi)$ and apply~$\sigma_\alpha$ to some point $(r,\theta,z)$ with $\theta \in (0,\, 2\alpha-2\pi)$, we obtain $(r, \, 2\alpha - \theta, \, z)$ where $2\alpha - \theta > 2\pi$. The last point, as an element of~$\cal M$, is distinct from the point $(r, \, 2\alpha - \theta - 2\pi, \, z)$. We denote still by~$x$ an arbitrary point $(r,t,z) \in \cal M$, for shortness, and we define the function $d \colon {\cal M} \to (0,+\infty)$ by letting $d(x) = r$. Thus, any function $u$ defined in $\Omega_{0\beta} \cup \,\Gamma_{\!0} \cup \,\Gamma_{\!\beta}$ can be extended to the submanifold~$\Omega_{-\beta,2\beta}$ by letting $u(x) := u(\sigma_0(x))$ for $x \in \Omega_{-\beta,0}$ and $u(x) := u(\sigma_\beta(x))$ for $x \in \Omega_{\beta,2\beta}$. If $u$ is a solution to~\eqref{P} and we extend it as above, the extended function (still denoted by~$u$) satisfies $-\Delta u = f(d(x),z,u)$ not only in~$\Omega_{0\beta}$, but also in~$\Omega_{-\beta,0}$ and in~$\Omega_{\beta,2\beta}$. To be precise, here $\Delta$ denotes the Beltrami-Laplace operator on~$\cal M$, which reduces to the standard Laplacian because $\cal M$ is flat. Moreover, since $f$ is H\"older continuous, standard regularity results imply that $u$ is regular through $\Gamma_{\!0}$ and~$\Gamma_{\!\beta}$, and satisfies the equation in all of~$\Omega_{-\beta,2\beta}$. Hence for every $\alpha \in (0,\beta)$, the function $w_\alpha \colon \Omega_{0\beta} \to \R$ given by $w_\alpha (x):=u(\sigma_\alpha(x))-u(x)$ is well defined and satisfies
\begin{equation}\label{eq:diff}
-\Delta w_\alpha -c_\alpha(x) \, w_\alpha =0 \ \text{ in } \Omega_{0\beta},
\end{equation}
where
\[c_\alpha (x):=\int_0^1 f' \big(d(x), \, z, \, t \, u(\sigma_\alpha(x)) + (1-t) \, u(x) \big) \, dt,\]
together with
\begin{equation}\label{boundary_conditions}
w_\alpha=0 \text{ on } \gamma_{0\beta} \cup \,\Gamma_{\!\alpha}
.
\end{equation}
In the sequel, the monotonicity of a solution~$u$ of~\eqref{P} will be detected through the sign of~$w_\alpha$: this motivates our interest in the function spaces $H^1_{\gamma_{0\alpha} \cup \,\Gamma_{\!\alpha}}(\Omega_{0\alpha})$ and $H^1_{\gamma_{\alpha\beta} \cup \,\Gamma_{\!\alpha}}(\Omega_{\alpha\beta})$. We have:

\begin{lemma}[Splitting lemma]\label{lem-1}
Let $u \in L^\infty(\Omega_{0\beta})$. Then for every $\alpha \in (0,\beta)$ we have
$$
\lambda_2(H^1_{\gamma_{0\beta}}(\Omega_{0\beta}))
\le
\max\Big\{\,
\lambda_1(H^1_{\gamma_{0\alpha} \cup \,\Gamma_{\!\alpha}}(\Omega_{0\alpha})),\
\lambda_1(H^1_{\gamma_{\alpha\beta} \cup \,\Gamma_{\!\alpha}}(\Omega_{\alpha\beta}))
\,\Big\}
.
$$
\end{lemma}
\begin{proof}
Let $\varphi_{0\alpha},\varphi_{\alpha\beta}$ be the eigenfunctions associated to the eigenvalues $\lambda_1(H^1_{\gamma_{0\alpha} \cup \,\Gamma_{\!\alpha}}(\Omega_{0\alpha}))$ and $\lambda_1(H^1_{\gamma_{\alpha\beta} \cup \,\Gamma_{\!\alpha}}(\Omega_{\alpha\beta}))$, respectively. Without loss of generality we assume $\|\varphi_{0\alpha}\|_{\Omega_{0\alpha}} = \|\varphi_{\alpha\beta}\|_{\Omega_{\alpha\beta}} = 1$, where $\| \cdot \|_D$ denotes the norm in $L^2(D)$. Since $\varphi_{0\alpha},\varphi_{\alpha\beta}$ satisfy a Dirichlet boundary condition on~$\Gamma_{\!\alpha}$, we may extend them to the whole~$\Omega_{0\beta}$ by letting $\varphi_{0\alpha} := 0$ in~$\Omega_{\alpha\beta}$ and $\varphi_{\alpha\beta} := 0$ in~$\Omega_{0\alpha}$. The extended functions are orthogonal in $H^1_{\gamma_{0\beta}}(\Omega_{0\beta})$ since their supports are disjoint, and we may define the two-dimensional linear subspace $W_0 = \{\, \phi \in  H^1_{\gamma_{0\beta}}(\Omega_{0\beta}) : \phi = a \, \varphi_{0\alpha} + b \, \varphi_{\alpha\beta} \mbox{ for } a,b \in \R \,\}$. Moreover, for every $\phi = a \, \varphi_{0\alpha} + b \, \varphi_{\alpha\beta} \not\equiv 0$ we have
\begin{align*}
\frac{\, Q_{\Omega_{0\beta}}(\phi) \,}
{\|\phi\|^2_{\Omega_{0\beta}}}
&=
\frac{\, 
a^2 \, Q_{\Omega_{0\alpha}}(\varphi_{0\alpha})
+
b^2 \, Q_{\Omega_{\alpha\beta}}(\varphi_{\alpha\beta})
\,}
{a^2 + b^2}
\\
\noalign{\smallskip}
&\le
\max\Big\{\,
Q_{\Omega_{0\alpha}}(\varphi_{0\alpha})
,\
Q_{\Omega_{\alpha\beta}}(\varphi_{\alpha\beta})
\,\Big\}
\\
\noalign{\bigskip}
&=
\max\Big\{\,
\lambda_1(H^1_{\gamma_{0\alpha} \cup \, \Gamma_{\!\alpha}}(\Omega_{0\alpha}))
,\
\lambda_1(H^1_{\gamma_{\alpha\beta} \cup \, \Gamma_{\!\alpha}}(\Omega_{\alpha\beta}))
\,\Big\}
.
\end{align*}
The variational formulation of the eigenvalues (see \cite[Theorem 1.42 (iii)]{Damascelli-Pacella}), sometimes called Poincar\'e's minimax characterization, ensures that 
\[
\lambda_2(H^1_{\gamma_{0\beta}}(\Omega_{0\beta})))
=\inf_{\stackrel{W\subset H^1_{\gamma_{0\beta}}(\Omega_{0\beta})}{{\rm dim}(W)=2}} \ \max_{\phi\in W \setminus \{0\}}\frac {\, Q_{\Omega_{0\beta}}(\phi) \,}{\|\phi\|^2_{\Omega_{0\beta}}},\]
and the conclusion follows immediately.
\end{proof}

\begin{corollary}\label{corollary}
Assume $\lambda_2(H^1_{\gamma_{0\beta}}(\Omega_{0\beta}))\geq 0$ for some $u \in L^\infty(\Omega_{0\beta})$, and let\/ $\alpha \in (0,\beta)$. Then either\/ $\lambda_1(H^1_{\gamma_{0\alpha} \cup \,\Gamma_{\!\alpha}}(\Omega_{0\alpha}))\geq 0$ or\/ $\lambda_1(H^1_{\gamma_{\alpha\beta} \cup \,\Gamma_{\!\alpha}}(\Omega_{\alpha\beta}))\geq 0$.
\end{corollary}

The positivity of the first eigenvalue is a necessary and sufficient condition in order that the weak maximum principle holds: see, for instance, \cite[Theorem~1.50]{Damascelli-Pacella}. When the first eigenvalue is allowed to vanish, we have a sign-preservation property for weak supersolutions of the Dirichlet problem associated to the operator $L_u$, as well as to the corresponding mixed problem. A weak supersolution of the Dirichlet problem for $L_u$ in~$\Omega_{0\alpha}$ is a function $v \in H^1(\Omega_{0\alpha})$ such that $v \ge 0$ almost everywhere on $\partial \Omega_{0\alpha}$ and the inequality
\begin{equation}\label{weak_Neumann}
\int_{\Omega_{0\alpha}}
\Big( \nabla v \, \nabla \varphi
-
f'(d(x), \, z, \, u) \, v \, \varphi \Big) dx
\ge 0
\end{equation}
holds for every nonnegative $\varphi \in H^1_0(\Omega_{0\alpha})$. Recall, further, that a weak supersolution of the mixed problem for $L_u$ in $\Omega_{0\alpha}$, with Neumann condition on~$\Gamma_{\! 0}$, is a function $v \in H^1(\Omega_{0\alpha})$ such that $v \ge 0$ a.e.\ on $\gamma_{0\alpha} \cup \Gamma_{\! \alpha}$ and the inequality~\eqref{weak_Neumann} holds for every nonnegative $\varphi \in H^1_{\gamma_{0\alpha} \cup \Gamma_{\! \alpha}}(\Omega_{0\alpha})$ \cite[p.~15]{Damascelli-Pacella}. A smooth function~$v$ satisfying $L_u \, v \ge 0$ pointwise in~$\Omega_{0\alpha}$ together with $v \ge 0$ on $\gamma_{0\alpha} \cup \Gamma_{\! \alpha}$ and $\partial v / \partial \nu \ge 0$ on~$\Gamma_{\! 0}$ is a also a weak supersolution of the mixed problem.

\begin{lemma}[Sign preservation]\label{sign}
Assume $\lambda_1(H^1_{\gamma_{0\alpha} \cup \,\Gamma_{\!\alpha}}(\Omega_{0\alpha})) \allowbreak \ge 0$ for some $u \in L^\infty(\Omega_{0\beta})$ and $\alpha \in (0,\beta)$. Let\/ $v \in C^1(\Omega_{0\alpha} \cup \Gamma_{\! 0}) \cap H^1(\Omega_{0\alpha})$ be a weak supersolution of the Dirichlet problem associated to the operator~$L_u$ in~$\Omega_{0\alpha}$, or a weak supersolution of the corresponding mixed problem with Neumann condition on~$\Gamma_0$.
\begin{enumerate}
\item If\/ $v$ satisfies $v = 0$ a.e.\ on\/ $\Gamma_{\!\alpha}$ without being identically zero in\/~$\Omega_{0\alpha}$, then either\/ $v > 0$ in\/~$\Omega_{0\alpha}$ or\/ $v < 0$ in\/~$\Omega_{0\alpha}$.
\item If\/ $v > 0$ a.e.\ on\/ $\Gamma_{\!\alpha}$, then $v > 0$ in\/~$\Omega_{0\alpha}$.
\end{enumerate}
\end{lemma}

\begin{proof}
The negative part $v^-(x) = \min\{\, v(x),0 \,\} \le 0$ satisfies $v^- = 0$ a.e.\ on~$\gamma_{0\alpha} \cup \Gamma_{\! \alpha}$. Furthermore, if $v$ is a weak supersolution of the Dirichlet problem, then $v^- = 0$ on~$\Gamma_{\! 0}$ as well. In both cases, we may take $\varphi = - v^-$ in~\eqref{weak_Neumann} and get $Q_{\Omega_{0\alpha}}(v^-) \le 0$. Since
\begin{equation}\label{f3}0 \le \lambda_1(H^1_{\gamma_{0\alpha} \cup \,\Gamma_{\!\alpha}}(\Omega_{0\alpha}))=\inf\left\{\left.\frac {Q_{\Omega_{0\alpha}}(\varphi)}{\|\varphi\|^2_{\Omega_{0\alpha}}}\ \right|\ \varphi\in H^1_{\gamma_{0\alpha}\cup \,\Gamma_{\!\alpha}}(\Omega_{0\alpha}) \setminus \{0\} \right\}\end{equation}
we deduce that either $v^-$ vanishes identically in~$\Omega_{0\alpha}$, or we may let $\varphi = v^-$ to minimize the quotient in \eqref{f3}. In the last case $v^-$ is a first eigenfunction of~$L_u$ in~$H^1_{\gamma_{0\alpha} \cup \,\Gamma_{\!\alpha}}(\Omega_{0\alpha})$ and therefore we must have $v^-<0$ in all of~$\Omega_{0\alpha}$ (see \cite[Theorem~1.42 (vi)]{Damascelli-Pacella}). If, instead, $v^-$ vanishes identically, then $v \ge 0$ in~$\Omega_{0\alpha}$ and by the strong maximum principle (see~\cite[Theorem~1.28]{Damascelli-Pacella}) we have either $v \equiv 0$ or $v > 0$ in~$\Omega_{0\alpha}$. The two claims follow at once.
\end{proof}

The starting point of our rotating plane argument, developed in Section~\ref{proof}, is the case when $\alpha = \frac\beta2$. The function~$w_\frac\beta2$, which is odd with respect to~$\Gamma_{\!\frac\beta2}$, satisfies $w_\frac\beta2 = 0$ on~$\gamma_{0\beta} \cup \, \Gamma_{\!\frac\beta2}$. Therefore the following corollary applies:

\begin{corollary}\label{cor-3}
Assume $u \in \cal X$ is a solution to \eqref{P} such that\/ $\lambda_2(H^1_{\gamma_{0\beta}}(\Omega_{0\beta}))\geq 0$ and assume $f(r,z,u)$ is convex with respect to~$u$. If\/ $w_{\frac \beta 2}$ does not vanish identically in~$\Omega_{0\beta}$, then either\/ $w_{\frac \beta 2} > 0$ in\/~$\Omega_{0\frac\beta2}$ or\/ $w_{\frac \beta 2} > 0$ in\/ $\Omega_{\frac\beta2\beta}$.
\end{corollary}
\begin{proof}
By Corollary \ref{corollary} we have that one among $\lambda_1(H^1_{\gamma_{0\frac\beta2} \cup \,\Gamma_{\!\frac\beta2}}(\Omega_{0\frac\beta2}))$ and $\lambda_1(H^1_{\gamma_{\frac\beta2\beta} \cup \,\Gamma_{\!\frac\beta2}}(\Omega_{\frac\beta2\beta}))$ is nonnegative. Suppose that $\lambda_1(H^1_{\gamma_{0\frac\beta2} \cup \,\Gamma_{\!\frac\beta2}}(\Omega_{0\frac\beta2})) \ge 0$, the other case being analogous. By the convexity of~$f$, the function $v = w_\frac\beta2 \in C^2(\Omega_{0\frac\beta2} \cup \Gamma_{\! 0} \cup \Gamma_{\! \frac\beta2}) \cap C^0(\overline \Omega_{0\frac\beta2}) \cap H^1(\Omega_{0\frac\beta2})$ satisfies $L_u \, v \ge 0$ in~$\Omega_{0\frac\beta2}$. Moreover $v$~vanishes on $\gamma_{0\frac\beta2} \cup \Gamma_{\! \frac\beta2}$ and satisfies $\frac{\partial v}{\partial \nu} = -\frac{\partial v}{\partial \theta} = 0$ on $\Gamma_{\! 0}$. Taking into account that $w_\frac\beta2$ is odd with respect to~$\Gamma_{\!\frac\beta2}$, the conclusion follows by letting $\alpha = \frac\beta2$ in Claim~1 of Lemma~\ref{sign}.\end{proof}

\section{Proof of Theorem~\ref{teo-1}}\label{proof}

The proof of Theorem~\ref{teo-1} is based on a fine interplay of the derivative $u_\theta$ and the function~$w_\alpha$ introduced in Section~\ref{se:2}. In fact, the strict monotonicity of~$u$ with respect to~$\theta$ in~$\Omega_{0\beta}$ holds if and only if $w_\alpha > 0$ in~$\Omega_{0\alpha}$ for every $\alpha \in (0,\beta)$. By contrast, symmetry of~$u$ with respect to~$\Gamma_{\!\frac\beta2}$, and constancy with respect to~$\theta$ as a special case, occur when $w_\frac\beta2$ vanishes identically. First of all we observe

\begin{proposition}\label{solution}
If\/ $u \in C^2(\Omega_{0\beta})$ satisfies the equation in~\eqref{P} pointwise, then the derivative~$u_\theta$ still belongs to $C^2(\Omega_{0\beta})$ and is a classical solution to
\begin{equation}\label{L=0}
L_u \, u_\theta = 0 \ \ \text{ in }\Omega_{0\beta}.
\end{equation}
\end{proposition}

\begin{proof}
As usual, we call \textit{test function} any smooth function $\varphi \in C^\infty(\Omega_{0\beta})$ compactly supported in~$\Omega_{0\beta}$. The set of all test functions is denoted by $ C^\infty_0(\Omega_{0\beta})$. To prove the claim, we differentiate the equality
$$
\int_{\Omega_{0\beta}} \nabla u \, \nabla \varphi \, dx
=
\int_{\Omega_{0\beta}}
\textstyle
f(d(x),z,u)
\, \varphi \, dx
$$
with respect to $\theta$. Since $\varphi_\theta$ is still a test function, the terms containing $\varphi_\theta$ cancel each other, and we obtain
\begin{equation}\label{weak_solution}
\int_{\Omega_{0\beta}} \nabla u_\theta \, \nabla \varphi \, dx
=
\int_{\Omega_{0\beta}}
\textstyle
f'(d(x),z,u)
\,
u_\theta
\, \varphi \, dx
.
\end{equation}
Hence $u_\theta$ is a weak solution of $L_u \, u_\theta = 0$. Since $f'$ is H\"older continuous by assumption, $u_\theta$ is a classical solution, as claimed. 
\end{proof}

To proceed further, observe that by Corollary~\ref{cor-3} the function~$w_{\frac \beta 2}(x)$ does not change sign in\/ $\Omega_{0\frac\beta2}$ nor in\/ $\Omega_{\frac\beta2\beta}$. We consider the case when $w_{\frac \beta 2}(x)\equiv 0$ in\/ $\Omega_{0\beta}$ first:

\begin{lemma}\label{lem-3}
Assume $u \in \cal X$ solves \eqref{P} and satisfies $\lambda_2(H^1_{\gamma_{0\beta}}(\Omega_{0\beta}))\geq 0$. If\/ $w_{\frac \beta 2}(x)\equiv 0$ in~$\Omega_{0\beta}$ then $u$~is constant with respect to\/~$\theta$.
\end{lemma}

\begin{proof}
If $w_{\frac \beta 2}(x)\equiv 0$ then $u$ is symmetric with respect to $\Gamma_{\!\frac\beta2}$ and the linearized operator~$L_u$ is invariant under reflections about $\Gamma_{\!\frac\beta2}$. In particular, the first eigenfunction $\varphi_1$ in $H^1_{\gamma_{0\beta}}(\Omega_{0\beta})$ is symmetric with respect to $\Gamma_{\!\frac\beta2}$. Furthermore, since $u_\theta \in H^1_0(\Omega_{0\beta})$, we may let $\varphi$ in~\eqref{weak_solution} approach~$u_\theta$ in~$H^1(\Omega_{0\beta})$ and deduce $Q_{\Omega_{0\beta}}(u_\theta) = 0$. In order to prove the lemma it is enough to show that $u_\theta$ vanishes identically in~$\Omega_{0\beta}$. The argument is by contradiction. Observe that $\varphi_1 \perp u_\theta$, i.e., $\varphi_1$ is orthogonal to~$u_\theta$ in $L^2(\Omega_{0\beta})$, because $u_\theta$ is odd with respect to~$\Gamma_{\!\frac\beta2}$. Furthermore, by assumption we have
$$
0
\le
\lambda_2(H^1_{\gamma_{0\beta}}(\Omega_{0\beta}))
=
\inf\left\{\left.\frac{\, Q_{\Omega_{0\beta}}(\varphi) \,}{\|\varphi\|^2_{\Omega_{0\beta}}}\ \right|\ \varphi_1 \perp \varphi\in H^1_{\gamma_{0\beta}}(\Omega_{0\beta}) \setminus \{0\} \right\}.
$$
If $u_\theta \not \equiv 0$, we see that the infimum is achieved at $\varphi = u_\theta$ because $Q_{\Omega_{0\beta}}(u_\theta) = 0$. But then $u_\theta$, being a minimizer of the Rayleigh quotient in $H^1_{\gamma_{0\beta}}(\Omega_{0\beta})$ under the constraint $\varphi_1 \perp \varphi$, should satisfy the Neumann conditions $\partial u_\theta / \partial \nu = 0$ on $\Gamma_{\!0} \cup \Gamma_{\!\beta}$. If this were the case, we could extend $u_\theta$ to $\Omega_{-\beta,\beta}$ by letting
\begin{equation}\label{extension}
\tilde u_\theta(x)
=
\begin{cases}
u_\theta(x), &x \in \Omega_{0\beta} \cup \Gamma_{\!0};
\\
\noalign{\medskip}
0, &x \in \Omega_{-\beta,0},
\end{cases}
\end{equation}
thus obtaining a function $\tilde u_\theta \in C^1(\Omega_{-\beta,\beta})$ that satisfies $L_u \, \tilde u_\theta = 0$ in the weak sense in~$\Omega_{-\beta,\beta}$ and vanishes identically in $\Omega_{-\beta,0}$. Since we are assuming that $u_\theta$ does not vanish identically in~$\Omega_{0\beta}$, this is in contrast with the unique continuation property \cite[p.~519]{Simon}, and the lemma follows.
\end{proof}

Next we consider the case when $w_{\frac \beta 2}> 0$ in $\Omega_{0\frac \beta 2}$. The case when $w_{\frac \beta 2}>0$ in $\Omega_{\frac \beta 2 \beta}$ is similar. Taking Corollary~\ref{corollary} into account, we may assume $\lambda_1(H^1_{\gamma_{0\frac\beta2} \cup \,\Gamma_{\!\frac\beta2}}(\Omega_{0\frac\beta2}))\geq 0$ without loss of generality. Let us show that the inequality $u_\theta > 0$ holds:

\begin{lemma}\label{lem-4}
Assume $u \in \cal X$ is a solution to~\eqref{P} such that\/ $\lambda_1(H^1_{\gamma_{0\frac\beta2} \cup \,\Gamma_{\!\frac\beta2}}(\Omega_{0\frac\beta2}))\ge 0$. If\/ $w_{\frac \beta 2}(x)>0$ in\/~$\Omega_{0\frac \beta 2}$ then
\begin{equation}\label{u_theta>0}
u_\theta > 0\ \ \text{in\/ } \Omega_{0\frac \beta 2}\cup \Gamma_{\!\frac \beta 2}.
\end{equation}
\end{lemma}
\begin{proof}
Arguing as in the proof of Lemma~\ref{lem-3} we see that the angular derivative $u_\theta$ (which cannot be odd, now, because $u$ is not symmetric about $\Gamma_{\!\frac\beta2}$) satisfies~\eqref{L=0} together with the boundary condition $u_\theta=0$ on $\gamma_{0\beta}\cup \Gamma_0\cup\Gamma_\beta$. Hence $u_\theta \in H^1_0(\Omega_{0\beta}) \cap C^1(\Omega_{0\beta} \cup \Gamma_{\! 0} \cup \Gamma_{\! \beta})$. Let us prove that $u_\theta>0$ on $\Gamma_{\!\frac\beta2}$. 
By~\eqref{eq:diff}, $w_\frac \beta2$ satisfies $-\Delta w_\frac \beta2-c_\frac \beta2(x) \, w_\frac \beta2=0$ in $\Omega_{0\frac \beta2}$.
Since $w_{\frac \beta 2}(x)>0$ in $\Omega_{0\frac \beta 2}$ by assumption, and $w=0$ on~$\Gamma_{\!\frac \beta 2}$ by definition, the Hopf boundary lemma \cite[Theorem~1.28]{Damascelli-Pacella} implies that $\frac{\partial w}{\partial \nu}<0$ on $\Gamma_{\!\frac \beta 2}$. Here $\nu$ is the outward unit normal to~$\Omega_{0\frac\beta2}$ at~$\Gamma_{\!\frac \beta 2}$. This means
\begin{equation}\label{negative}
\left.\frac{\partial w}{\partial \nu}\right|_{\Gamma_{\!\frac \beta 2}}=-2\left.\frac{\partial u}{\partial \nu}\right|_{\Gamma_{\!\frac \beta 2}}=-2\left.\frac{\partial u}{\partial \theta}\right|_{\Gamma_{\!\frac \beta 2}}<0.
\end{equation}
Hence $u_\theta > 0$ on~$\Gamma_{\!\frac\beta2}$, as claimed, and the conclusion follows from Claim~2 of Lemma~\ref{sign}.
\end{proof}

We are now ready to apply the rotating plane method.

\begin{proof}[Proof of Theorem~\ref{teo-1}]
We assume that $\lambda_2(H^1_{\gamma_{0\beta}}(\Omega_{0\beta})) \ge 0$ and prove that either $u_\theta \equiv 0$ or $u_\theta$~keeps its sign in~$\Omega_{0\beta}$. In the last case, from~\eqref{L=0} and the boundary condition it follows immediately that $u_\theta$ is a first Dirichlet eigenfunction of~$L_u$ in~$\Omega_{0\beta}$ and therefore $\lambda_1(H^1_0(\Omega_{0\beta})) = 0$. At the end of the proof, a final observation shows that if $\lambda_1(H^1_{\gamma_{0\beta}}(\Omega_{0\beta})) \ge 0$ then $u$ is constant in~$\theta$, which completes the proof of the theorem.

If $w_{\frac \beta 2}\equiv 0$ in\/ $\Omega_{0\beta}$ then $u$ is independent of~$\theta$ by Lemma~\ref{lem-3}. Otherwise, by Corollary~\ref{cor-3}, either $w_\frac\beta2 > 0$ in~$\Omega_{0\frac\beta2}$ or $w_\frac\beta2 > 0$ in~$\Omega_{\frac\beta2\beta}$. By Corollary~\ref{corollary} we may assume $\lambda_1(H^1_{\gamma_{0\frac\beta2} \cup \,\Gamma_{\!\frac\beta2}}(\Omega_{0\frac\beta2}))\geq 0$ and $w_{\frac \beta 2} > 0$ in $\Omega_{0\frac \beta 2}$ without loss of generality (indeed, if $w_\frac\beta2 < 0$ in~$\Omega_{0\frac\beta2}$ the argument is similar). Let us prove that $u_\theta > 0$ in $\Omega_{0\beta}$. Denote by $A'$ the set of all $\alpha' \in [\frac\beta2,\beta]$ having the property that for every $\alpha \in (0,\alpha')$ the function $w_{\alpha}$ is positive in~$\Omega_{0\alpha}$.

\textit{Part I.} The set~$A'$ contains $\frac\beta2$ and therefore it is not empty. To see this, we have to prove that $u(x) < u(\sigma_\alpha(x))$ for every $\alpha \in (0,\frac\beta2)$ and $x \in \Omega_{0\alpha}$, which is equivalent to say that $u(r,\theta_1,z) < u(r,\theta_2,z)$ for $0 < \theta_1 < \theta_2$ such that $\theta_1 + \theta_2 < \beta$. If $\theta_2 \le \frac\beta2$ the conclusion follows from~\eqref{u_theta>0}. Otherwise we must have $\theta_1 < \beta-\theta_2 < \frac\beta2 < \theta_2$ and we argue as follows. Let $x_i$, for $i = 1,2$, be the point whose cylindrical coordinates are $(r,\theta_i,z)$, and observe that the coordinates of $\sigma_\frac\beta2(x_2)$ are $(r, \, \beta - \theta_2, \, z)$. We may write $u(\sigma_\frac\beta2(x_2)) < u(x_2)$ because we are considering $w_\frac\beta2 > 0$ in~$\Omega_{0\frac\beta2}$ (hence $w_\frac\beta2 < 0$ in~$\Omega_{\frac\beta2\beta}$). Furthermore we have $u(x_1) < u(\sigma_\frac\beta2(x_2))$ because of~\eqref{u_theta>0}, and therefore $\frac\beta2 \in A'$ as claimed. We note for later usage that the preceding argument still holds when $\theta_1 = 0$, hence
\begin{equation}\label{still}
\mbox{$w_\alpha > 0$ on $\Gamma_{\!0}$ for every $\alpha \in (0,\frac\beta2)$.}
\end{equation}

\textit{Part II.} The set $A'$ is a closed subinterval of $[\frac\beta2,\beta]$. Indeed, by the definition of $A'$ it follows that if $\alpha',\alpha''$ satisfy $\frac\beta2 \le \alpha' \le \alpha''$ and $\alpha'' \in A'$ then $\alpha' \in A'$, hence $A'$ is an interval. It also follows immediately that $A'$ is closed, and we may write $A' = [\frac\beta2,\tilde\alpha]$ for some $\tilde\alpha \in [\frac\beta2,\beta]$.

\textit{Part III.} The derivative $u_\theta$ is positive in the sector $\Omega_{0\tilde\alpha}$. To check this, we verify that $u_\theta > 0$ along~$\Gamma_{\!\alpha}$ for each $\alpha \in (0,\tilde\alpha)$. The argument is the same as in Lemma~\ref{lem-4}: since $w_\alpha$ is positive in~$\Omega_{0\alpha}$ by the definition of~$A'$, and satisfies~\eqref{eq:diff}, by the Hopf boundary-point lemma it must also satisfy $\partial w_\alpha / \partial \nu < 0$ along~$\Gamma_{\!\alpha}$, where $w_\alpha = 0$. Hence $u_\theta > 0$ along~$\Gamma_{\!\alpha}$, as claimed.

\textit{Part IV.} If $\tilde\alpha < \beta$ then $w_{\tilde\alpha} > 0$ in~$\Omega_{0\tilde\alpha}$. Indeed, by continuity we may write $w_{\tilde\alpha} \ge 0$ in~$\Omega_{0\tilde\alpha}$, and by the convexity of~$f$ it follows that $w_{\tilde\alpha}$ satisfies the inequality $L_u \, w_{\tilde\alpha} \ge 0$. Hence either $w_{\tilde\alpha} > 0$ or $w_{\tilde\alpha} \equiv 0$ in~$\Omega_{0\tilde\alpha}$ by the strong maximum principle. Let us exclude the second case. If $w_{\tilde\alpha} \equiv 0$ in~$\Omega_{0\tilde\alpha}$ then $\tilde\alpha > \frac\beta2$ because we are assuming $w_\frac\beta2 > 0$ in~$\Omega_{0\frac\beta2}$. Furthermore, by differentiating the equality $u(r, \, 2\tilde\alpha - \theta, \, z) = u(r,\theta,z)$ we obtain $u_\theta(r, \, 2\tilde\alpha - \theta, \, z) = -u_\theta(r,\theta,z)$ for $(r,\theta,z) \in \Omega_{0\tilde\alpha}$, which yields at $\tilde\theta = 2\tilde\alpha - \beta \in (0,\tilde\alpha)$
$$
0
=
\left.\frac{\partial u}{\partial \nu}\right|_{\Gamma_{\!\beta}}
=
u_\theta(r,\beta,z)
=
-u_\theta(r,\tilde\theta,z)
,
$$
contradicting Part~III. Hence we can write the implication $\tilde\alpha < \beta \implies w_{\tilde\alpha} > 0$  in~$\Omega_{0\tilde\alpha}$.

\textit{Part V.} Recall that for every $\alpha \in (0,\beta)$ the function $w_\alpha$ is well defined in~$\Omega_{0\alpha}$ and satisfies \eqref{eq:diff}-\eqref{boundary_conditions} together with~\eqref{still}. To complete the boundary conditions, let us check that
\begin{equation}\label{complete}
\mbox{$w_\alpha > 0$ on~$\Gamma_{\!0}$ for every $\alpha \in [\frac\beta2,\beta)$,}
\end{equation}
which is equivalent to $u(r,0,z) < u(r,2\alpha,z)$. By the definition of the extended function~$u$, we have $u(r,2\alpha,z) = u(r, \, 2\beta - 2\alpha, \, z)$. The point $(r, \, 2\beta - 2\alpha, \, z)$ can be rewritten as $\sigma_{\alpha'}(r,0,z)$ by choosing $\alpha' = \beta - \alpha > 0$. If $\alpha \in (\frac\beta2,\beta)$ then $\alpha' < \frac\beta2$ and the inequality $u(r,0,z) < u(\sigma_{\alpha'}(r,0,z))$ follows from~\eqref{still}. Finally, since $w_\frac\beta2$ is positive in $\Omega_{0\frac\beta2}$ by assumption and satisfies $\partial w_\frac\beta2 /\partial \nu = 0$ on~$\Gamma_{\! 0}$, the Hopf boundary point lemma prevents it from vanishing  there.

\textit{Part VI.} The second endpoint $\tilde\alpha \ge \frac\beta2$ of the interval~$A'$ is in fact~$\beta$. To prove this, we show that if $\tilde \alpha < \beta$ then for some $\varepsilon_0 > 0$ and for every $\alpha \in [\tilde \alpha, \, \tilde \alpha + \varepsilon_0)$ we have $w_\alpha > 0$ in~$\Omega_{0\alpha}$, which contradicts the definition of~$\tilde \alpha$. In order to reach our goal we need a suitable maximum principle. Observe, to begin with, that the sup-norm $\Vert c_\alpha \Vert_\infty$ is bounded uniformly with respect to~$\alpha$. Therefore, by the weak maximum principle in small domains \cite[Theorem~1.20]{Damascelli-Pacella} there exists $\delta > 0$ (independent of~$\alpha$) such that if an open subset $D \subset \Omega_{0\alpha}$ satisfies $|D| < \delta$, and if $w_\alpha \ge 0$ on~$\partial D$, then $w_\alpha \ge 0$ in all of~$D$. Accordingly, let us fix a (nonempty) compact subset $K \subset \Omega_{0\tilde\alpha}$ such that $|\Omega_{0\tilde\alpha} \setminus K| < \frac\delta2$ and define
$$
\eta = \min_K w_{\tilde \alpha}
.
$$
Recall that $\eta > 0$ by Part~IV. Then, for a convenient $\varepsilon_0 > 0$ we achieve that for all $\alpha \in [\tilde \alpha, \allowbreak \, \tilde \alpha + \varepsilon_0)$ the open set $D_\alpha = \Omega_{0\alpha} \setminus K$ satisfies $|D_\alpha| < \delta$ and
$$
\mbox{$w_\alpha > \frac\eta2$ in~$K$.}
$$
Taking into account the boundary conditions established in Part~V, we may write $w_\alpha \ge 0$ on~$\partial D_\alpha$ and therefore, by the weak maximum principle in small domains, the inequality $w_\alpha \ge 0$ holds in~$D_\alpha$, and consequently in all of~$\Omega_{0\alpha}$. More precisely, since $w_\alpha > 0$ in~$K$, by the strong maximum principle we obtain $w_\alpha > 0$ in~$\Omega_{0\alpha}$ for every $\alpha\in [\tilde \alpha, \, \tilde \alpha + \varepsilon_0)$. This contradiction proves that $\tilde\alpha = \beta$, as claimed.

\textit{Conclusion.} By Part VI, the solution $u$ is strictly increasing with respect to~$\theta$ in~$\Omega_{0\beta}$, and by Part~III it satisfies~$u_\theta > 0$. We have thus proved that if $\lambda_2(H^1_{\gamma_{0\beta}}(\Omega_{0\beta})) \ge 0$ then either $u_\theta$ vanishes identically in~$\Omega_{0\beta}$ or it keeps its sign there. In order to complete the proof, it suffices to check that if $\lambda_1(H^1_{\gamma_{0\beta}}(\Omega_{0\beta})) \ge 0$ the last case cannot occur, i.e., $u_\theta$ does not keep its sign in~$\Omega_{0\beta}$. In fact, if this were the case, we would have $0 \le \lambda_1(H^1_{\gamma_{0\beta}}(\Omega_{0\beta})) \le \lambda_1(H^1_0(\Omega_{0\beta})) = 0$ by the comparison of Dirichlet and Neumann eigenvalues (see, for instance, \cite[Theorem~1, p.~408]{Courant-Hilbert} or \cite[Example 11.13]{Helffer}). This shows that $u_\theta$ would also minimize the Rayleigh quotient in~$H^1_{\gamma_{0\beta}}(\Omega_{0\beta})$ and therefore it should satisfy the Neumann boundary condition $\partial u_\theta / \partial \nu = 0$ on $\Gamma_{\!0} \cup \Gamma_{\!\beta}$. But then $u_\theta$ would extend as in~\eqref{extension} to a $C^1$-eigenfunction in~$\Omega_{-\beta,\beta}$ vanishing identically in $\Omega_{-\beta,0}$, and this is impossible by the unique continuation property.
\end{proof}

\section{Cylindrical domains}\label{cylindrical}

When the opening of a sector-like domain tends to zero, the domain is asymptotic to a cylinder. To manage with this case we have to arrange the notation. Changes are specific to the present section. Here we denote by $\omega$ a Lipschitz, bounded, $(N - 1)$-dimensional domain lying in the hyperplane $x_1 = 0$ and by $\Omega_{\theta_1\theta_2}$ the cylinder $(\theta_1,\theta_2) \times \omega$, where $\theta_1,\theta_2 \in \mathbb R$ satisfy $\theta_1 < \theta_2$. Given $\beta \in (0,+\infty)$, we let $\mathcal X$ be the set of functions $u\in C^2(\Omega_{0\beta})\cap C^0(\overline \Omega_{0\beta})\cap H^1( \Omega_{0\beta})$ such that $\frac{\partial u}{\partial x_1}\in H^1(\Omega_{0\beta})$ and we consider functions $u\in \mathcal X$ that satisfy, pointwise, the boundary-value problem:
\begin{equation}\label{CP}
\left\{\begin{array}{ll}
-\Delta u = f(z,u) \qquad & \text{ in } \Omega_{0\beta}, \\
\noalign{\medskip}
\frac{\partial u}{\partial x_1}=0 & \text{ on $\partial \Omega_{0\beta}$, }
\end{array} \right.
\end{equation}
where $z = (x_2, \ldots, x_N) \in \mathbb R^{N - 1}$ and $f(z,u)$ is locally H\"older continuous in $\mathbb R^N$ together with its derivative $f' = \partial f / \partial u$. 
Letting $\Gamma_{\!\alpha} = \{\alpha\} \times \omega$ and $\gamma_{\theta_1\theta_2} = [\theta_1,\theta_2] \times \partial \omega$, problem \eqref{CP} is equivalent to the following mixed boundary-value problem:
$$
\left\{\begin{array}{ll}
-\Delta u = f(z,u) \qquad & \text{ in } \Omega_{0\beta}, \\
\noalign{\medskip}
u=g(z) & \text{ on } \gamma_{0\beta},\\
\noalign{\medskip}
\frac{\partial u}{\partial \nu}=0 & \text{ on } \Gamma_{\! 0} \cup \Gamma_{\! \beta}
\end{array} \right.
$$
where $g$ is continuous on~$\gamma_{0\beta}$. The preceding results can be extended to such a case. We have:

\begin{theorem}
Let $u\in \mathcal X$ be a solution to \eqref{CP}, where $f(z,u)$ is convex with respect to~$u$.
\begin{enumerate}
\item If\/ $\lambda_1(H^1_{\gamma_{0\beta}}(\Omega_{0\beta}))\geq 0$ then $u$ is constant with respect to~$x_1$.
\item If\/ $\lambda_1(H^1_{\gamma_{0\beta}}(\Omega_{0\beta})) < 0 \le \lambda_2(H^1_{\gamma_{0\beta}}(\Omega_{0\beta}))$ then $u$ is either constant with respect to~$x_1$ or strictly monotone with respect to~$x_1$ in\/ $\Omega_{0\beta}$. In the last case, the derivative $u_1 = \partial u / \partial x_1$ does not vanish in\/~$\Omega_{0\beta}$, the first Dirichlet eigenvalue $\lambda_1(H^1_0(\Omega_{0\beta}))$ equals zero and $u_1$ is a corresponding eigenfunction.
\end{enumerate}
\end{theorem}

\noindent The argument is identical to the proof of Theorem~\ref{teo-1} found in Section~\ref{proof}, keeping in mind that in the present case the reflection $\sigma_\alpha$ is obviously defined as $\sigma_\alpha(x_1,z) = (2\alpha - x_1, \, z)$, and every occurrence of~$\theta$ has to be replaced with~$x_1$. Since $\gamma_{0\beta}$ and $\Gamma_{\alpha}$ intersects orthogonally we can obtain that $ \partial u / \partial x_1\in H^1(\Omega_{0\beta})$ using \cite[Proposition 6.1]{PacellaTralli} when $\omega$ is smooth.

\section{Examples}\label{exe}
In this section we produce some examples of solutions to \eqref{P} that are not constant in~$\theta$ and satisfy Claim~2 of Theorem \ref{teo-1}. In the first examples $\Omega_{0\beta}$ is a sector of the unit disc $B_1 \subset \R^2$.

\begin{example}
Let $N = 2$ and $\Omega=B_1$. We consider the eigenvalue problem
\begin{equation}\label{eigenvalue}
\left\{\begin{array}{ll}
-\Delta \psi = \lambda \psi \qquad & \text{ in } \Omega_{0\beta}, \\
\noalign{\medskip}
\psi= 0 & \text{ on } \gamma_{0\beta},\\
\noalign{\medskip}
\frac{\partial \psi}{\partial \nu}=0 & \text{ on } \Gamma_{\! 0} \cup \Gamma_{\! \beta}
\end{array} \right.
\end{equation}
for $0<\beta<2\pi$. If we neglect to give the eigenvalues an increasing order, the solutions $(\lambda,\psi)$ of~\eqref{eigenvalue} are given by 
\[\begin{array}{ll}
\psi_{nk}(r,\theta)=J_{s_n} \big(j_{s_n,k} \, r\big) \, \cos(s_n \, \theta)\\
\noalign{\medskip}
\lambda_{nk}=j^2_{s_n,k}
\end{array} 
\]
for $n=0,1,2,\dots$ and $k=1,2,\dots$ where $J_{s_n}$ is the Bessel function of order $s_n = \frac{n\pi}{\beta}$, and $j_{s_n,k}$ is the $k$-th positive zero of $J_{s_n}$. Each $\psi_{nk}$ can be considered as a solution of~\eqref{P} or \eqref{instance} where $f(r,z,u) = \lambda_{nk} \, u$, and the corresponding  linearized operator is $L_{\psi_{nk}} = -\Delta - \lambda_{nk} \, I$. Of course, if $\psi_{nk}$ is the $i$-th eigenfunction of problem~\eqref{eigenvalue} for some positive integer~$i$, then the $i$-th eigenvalue of $L_{\psi_{nk}}$ in $H^1_{\gamma_{0\beta}}(\Omega_{0\beta})$ vanishes. The first eigenvalue of problem~\eqref{eigenvalue} is $\lambda_{01}=j^2_{0,1}$ and the first eigenfunction is radial and given by
\[\psi_{01}(r,\theta)=J_0\big(j_{0,1} \, r\big).\]
Thus, we have an instance of problem~\eqref{P} or \eqref{instance} where $\lambda_1(H_{\gamma_{0\beta}}(\Omega_{0\beta})) = 0$ and a corresponding solution $u = \psi_{01}(r,\theta)$ which is radial, as stated in Theorem \ref{teo-1}, Claim {\em 1}. In order to find the second eigenvalue and eigenfunction of~\eqref{eigenvalue} we have to compare $j_{0,2}\approx 5.5201$ and $j_{\frac{\pi}{\beta},1}$ whose value depends on the angle $\beta$. It can be seen that if $\beta=\frac \pi 2$ then $J_\frac\pi\beta = J_2$ and $j_{\frac{\pi}{\beta},1}=j_{2,1}\approx 5.1356<j_{0,2}$, while if $\beta=\frac \pi 3$ then $J_\frac\pi\beta = J_3$ and $j_{\frac\pi\beta,1}=j_{3,1}\approx 6.3802>j_{0,2}$. Next, using that $j_{s,1}$ is a continuous and strictly increasing function of the variable $s$ (see for instance \cite[pp.~67-68]{elbert}), it follows that there exists a unique angle $\widehat \beta\in (\frac \pi 3,\frac \pi 2)$ such that $j_{\frac\pi{\widehat \beta},1}=j_{0,2}$. Moreover one has $j_{\frac{\pi}\beta,1}<j_{0,2}$ for $\beta > \widehat \beta$ while $j_{\frac\pi\beta,1} > j_{0,2}$ for $0 < \beta < \widehat \beta$. The value of $\widehat \beta$ can be estimated numerically obtaining $\widehat \beta\approx 1.3629$. Reassuming, for $\beta>\widehat \beta$ the second eigenvalue of~\eqref{eigenvalue} is $\lambda_{11}=j^2_{\frac \pi\beta,1}$, and the second eigenfunction
\[\psi_{11}(r,\theta)=J_{\frac\pi\beta}\big(j_{\frac \pi\beta,1} \, r\big) \, \cos \big(\textstyle\frac\pi\beta \, \theta\big)\]
is nonradial and strictly monotone in the angular variable in~$\Omega_{0\beta}$. Let us observe that the second eigenfunction $\psi_{11}$ satisfies the assumptions of Theorem \ref{teo-1}, Claim {\em 2} since the second eigenvalue of the linearized operator $L_{\psi_{11}} = -\Delta - j^2_{\frac\pi{\beta},1} \, I$ in $H^1_{\gamma_{0 \beta}}(\Omega_{0 \beta})$ vanishes. For $\beta < \widehat \beta$ the second eigenvalue of problem~\eqref{eigenvalue} is $\lambda_{02}=j^2_{0,2}$ and the second eigenfunction is radial and given by
\[\psi_{02}(r,\theta)=J_0(j_{0,2} \, r),\]
while for $\beta=\widehat \beta$ the second eigenvalue has multiplicity two and the eigenspace is spanned by $\psi_{02} = J_0(j_{0,2} \, r)$ and $\psi_{\widehat s 1} = J_{\widehat s}\big(j_{\widehat s,1} \, r\big) \, \cos \big(\widehat s \, \theta\big)$, where $\widehat s = \pi/\widehat\beta$. Thus, letting $m = j^2_{0,2} = j^2_{\widehat s,1}$ and $f(r,z,u) = mu$ we obtain an instance of problem~\eqref{P} or \eqref{instance} where $\lambda_2(H^1_{\gamma_{0\widehat \beta}}(\Omega_{0\widehat\beta})) = 0$ and there are both a radial solution and a solution strictly monotone in~$\theta$  thus implying that $\lambda_1(H^1_0(\Omega_{0\widehat\beta})=0$. Since we are in the linear case,  then $L_{\psi_{02}}= L_{\psi_{\widehat s 1}}$ showing that  even if $\frac{\partial \psi_{02}}{\partial \theta}\equiv 0$ the corresponding Dirichlet eigenvalue $\lambda_1(H^1_0(\Omega_{0\widehat\beta}))$ vanishes.

\begin{figure}[h]
  \centering
  \resizebox{7.5cm}{!}{\includegraphics{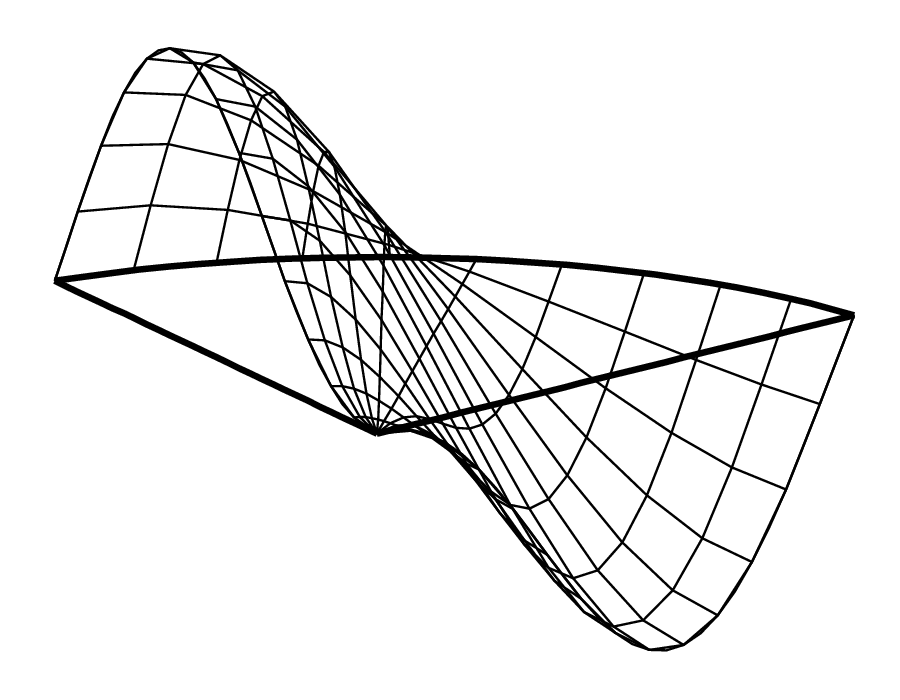}}
\quad
  \resizebox{7.5cm}{!}{\includegraphics{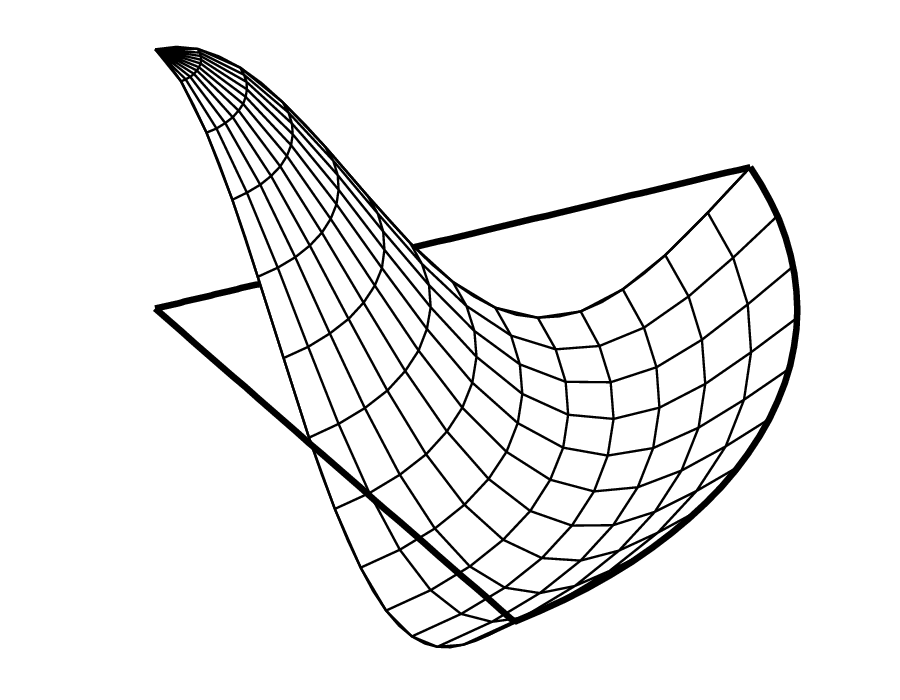}}
  \vspace{-0cm}\caption{A solution monotone in~$\theta$ (left) and a radial solution (right) in $\Omega_{0 \widehat\beta}$ from different viewpoints}
\end{figure}
Next we want to show that, if we allow the solution $u$ in \eqref{P} or in \eqref{instance} to have a higher Morse index, namely $\lambda_2(H^1_{\gamma_0\beta}(\Omega_{0\beta}))<0$ then $u$ can be nonradial without being strictly monotone in $\theta$ in $\Omega_{0\beta}$. Indeed 
if we consider the third eigenvalue and eigenfunction of \eqref{eigenvalue} in the case when $\beta >\widehat \beta$ we have to compare $j_{0,2}\approx 5.5201$ and $j_{\frac{2\pi}{\beta},1}$ (that corresponds to $n=2$) 
and it is easy to see that when $\beta=\pi$ then $j_{\frac{2\pi}{\beta},1}=j_{2,1}\approx 5.1356<j_{0,2}$. As before we can say that there exists $\widetilde \beta >\widehat \beta$ 
such that when $\beta>\widetilde \beta$ the third eigenvalue of~\eqref{eigenvalue} is $\lambda_{2,1}=j^2_{\frac{2\pi}{\beta},1}$ and
the third eigenfunction is 
\[
\psi_{21}(r,\theta)=J_{\frac {2\pi}\beta}\big(j_{\frac {2\pi}\beta,1} \, r\big) \, \cos \big(\textstyle\frac{2\pi}\beta \, \theta\big)
\]
which is symmetric with respect to the bisector of $\Omega_{0\beta}$, strictly monotone with respect to $\theta$ in $\Omega_{0\frac \beta 2}$ and in $\Omega_{\frac \beta 2 \beta}$ and satisfies $\lambda_3(H^1_{\gamma_{0\beta}}(\Omega_{0\beta}))=0$. \\
We conjecture that this behavior should be the only new one when one passes from solutions to \eqref{P} with $\lambda_2(H^1_{\gamma_{0\beta}}(\Omega_{0\beta})) \ge 0$ to solutions with $\lambda_2(H^1_{\gamma_{0\beta}}(\Omega_{0\beta}))<0\leq \lambda_3(H^1_{\gamma_{0\beta}}(\Omega_{0\beta}))$.
\end{example}

\begin{example}
Next we consider a semilinear equation. Let us recall the H\'enon problem
\begin{equation}\label{henon}
\left\{\begin{array}{ll}
-\Delta u = |x|^\alpha \, u^p,\ u > 0 \qquad & \text{ in } B_1, \\
\noalign{\medskip}
u = 0 & \text{ on } \partial B_1,\\
\end{array} \right.
\end{equation}
where $B_1$ is the unit ball in~$\mathbb R^N$, $\alpha>0$ and $p>1$. It has been proved in \cite{SmetsSuWillem} that for any $p>1$ there exists $\alpha^*(p)$ such that for any $\alpha>\alpha^*(p)$ the ground state of~\eqref{henon} is nonradial. When $p>2$, ground states have a symmetry axis (that we can assume, up to a rotation, is the $x_1$-axis) and they are increasing (or decreasing) in the polar angle from this axis, namely in~$\theta$ (see \cite{BartschWethWillem} and \cite{PacellaWeth}). If we call $\bar u$ the restriction to the half ball $\Omega_{0\pi} = B_1 \cap \{\, x_2 > 0 \,\}$ of a ground state~$u$ to~\eqref{henon}, we have that $\bar u$ satisfies 
\begin{equation}\label{??}
\left\{\begin{array}{ll}
-\Delta \bar u = |x|^\alpha \, \bar u^p \qquad & \text{ in } \Omega_{0\pi}, \\
\noalign{\medskip}
\bar u= 0 & \text{ on } \gamma_{0\pi},\\
\noalign{\medskip}
\frac{\partial \bar u}{\partial \nu}=0 & \text{ on } \Gamma_{\! 0}\cup \Gamma_{\! \pi}
\end{array} \right.
\end{equation}
and, when $\alpha$ is large enough, it is strictly monotone in $\theta$. Moreover $\bar u$ satisfies the assumption of Theorem \ref{teo-1}, namely $\lambda_2(H^1_{\gamma_{0\pi}}(\Omega_{0\pi}))\geq  0$. Indeed if $\lambda_2(H^1_{\gamma_{0\pi}}(\Omega_{0\pi}))<0$ then the linearized operator has two negative eigenvalues with corresponding eigenfunctions orthogonal in $L^2(\Omega_{0\pi})$. We can extend the eigenfunctions by symmetry about the hyperplane $\{\, x_2 = 0 \,\}$ thus obtaining two $L^2$-orthogonal eigenfunctions of the linearization to \eqref{henon} associated with negative eigenvalues. This is not possibile since the linearized operator associated to a ground state~$u$ can have only one negative eigenvalue, hence we must have $\lambda_2(H^1_{\gamma_{0\pi}}(\Omega_{0\pi}))\geq  0$. In the special case when $N = 2$, letting $v(\rho,\phi)= c \, \bar u(r,\theta)$ with $c=\big(\frac \pi \beta\big)^\frac2{p-1}$ for $r=\rho^\frac{\pi}\beta$ and $\theta=\frac \pi \beta \, \phi$, $\beta \in (0, 2\pi)$, a straightforward computation shows that $v$ solves
\[
\left\{\begin{array}{ll}
-\Delta v =  
|x|^{(\alpha+2)\frac \pi\beta-2} \, v^p \qquad & \text{ in } \Omega_{0\beta}, \\
\noalign{\medskip}
v= 0 & \text{ on } \gamma_{0\beta},\\
\noalign{\medskip}
\frac{\partial v}{\partial \nu}=0 & \text{ on } \Gamma_{\! 0} \cup \Gamma_{\! \beta}
\end{array} \right.
\]
giving an example of monotone solution in a sector of any amplitude $\beta$. 
This solution satisfies $\lambda_2(H^1_{\gamma_{0\beta}}(\Omega_{0\beta}))\geq  0$. 
Indeed if $\varphi_1(\rho,\phi)$ and $\varphi_2(\rho,\phi)$ are two $L^2(\Omega_{0\beta})$-orthogonal  functions that make negative the quadratic form associated with the linearization to this last equation at the solution $v$, then the functions $\bar \varphi_1(r,\theta)=\varphi_1(\rho,\phi)$ and $\bar \varphi_2(r,\theta)=\varphi_2(\rho,\phi)$ are two $L^2(\Omega_{0\pi})$-orthogonal  functions that makes negative the quadratic form associated with the linearization to \eqref{??} at the solution $\bar u$ and this is not possible.
\end{example}

\begin{example}
Here we consider the case of an annulus $A$ of $\R^2$. In \cite[Proposition 1.4]{Gladiali1} it has been proved that for any $k=1,2,\dots$, there exists an exponent $p_k>1$ such that the least energy solution of the Lane-Emden problem 
\[
\left\{\begin{array}{ll}
-\Delta u =  u^p \qquad & \text{ in } A, \\
\noalign{\medskip}
u= 0 & \text{ on } \partial A
\end{array} \right.
\]
in the space of $H^1$ functions that are $\frac {2\pi}k$ periodic in $\theta$ is nonradial when $p>p_k$. Moreover, by \cite{Gladiali2}, up to a rotation, these solutions are strictly monotone in $\theta$ in the sector $\Omega_{0\frac \pi k}$ and, by symmetry reasons, they satisfy a Neumann boundary condition on $\Gamma_{\! 0} \cup \Gamma_{\! \frac \pi k}$. Finally, since they are least energy solutions, they satisfy $\lambda_2(H^1_{\gamma_{0\frac \pi k}}(\Omega_{0\frac \pi k}))\geq  0$.
\end{example}

\section{Symmetry and monotonicity in unbounded domains}\label{unbounded}

Hereafter we consider the case when $\Omega_{0\beta}$ is not bounded. To keep the presentation simple and to deal with some difficulties coming from the unboudeness of the domain, we focus on the special case when the domain of the problem is a sector of the $N$-dimensional Euclidean space with an angular amplitude $\beta \in (0,2\pi)$. More generally, we also consider the set difference of such a sector and the infinite cylinder $d(x) \le r_1$, where $d(x)=\sqrt{x_1^2+x_2^2}$. To be precise, choosing $r_1 \in [0,+\infty)$ and $\theta_1,\theta_2 \in [0,2\pi)$, with $\theta_1 < \theta_2$, we denote by $\Omega_{\theta_1\theta_2}$ the domain
$$
\Omega_{\theta_1\theta_2}
=
\{\, x \in \R^N
:
d(x) > r_1,\ \theta(x) \in (\theta_1,\theta_2) \,\}
$$
whose boundary contains the open, flat subsets $\Gamma_{\!\theta_i} = \{\, x \in \R^N : d(x) > r_1,\ \theta(x) = \theta_i \,\}$, $i =1,2$. Moreover when $r_1>0$ we may write $\partial \Omega_{\theta_1\theta_2} =\Gamma_{\! \theta_1}\cup \Gamma_{\! \theta_2} \cup \gamma_{\theta_1\theta_2}$ where $\gamma_{\theta_1\theta_2}$
is the closed subset
$$
\gamma_{\theta_1\theta_2} := \{\,
x \in \R^N : d(x)=r_1,\ \theta(x) \in [\theta_1,\theta_2] \,\}.
$$
When $r_1=0$, instead, we have $\partial \Omega_{\theta_1\theta_2} =\Gamma_{\! \theta_1}\cup \Gamma_{\! \theta_2} \cup \Upsilon$ and $\gamma_{\theta_1\theta_2} = \emptyset$. In this part of the paper we deal with problem~\eqref{P} where the domain $\Omega_{0\beta}$ is unbounded and defined as above. The function $f(r,z,u)$ is subject to the same conditions stated in the Introduction. Let $\cal Y$ be the set of all functions $u \in C^2(\Omega_{0\beta}) \cap C^0(\overline \Omega_{0\beta})$ such that $u_\theta\in C^0(\overline \Omega_{0\beta}\setminus \Upsilon)$ and having a square-summable gradient $\nabla u \in L^2(\Omega_{0\beta},\mathbb R^N)$. Here we investigate functions $u \in \cal Y$ satisfying~\eqref{P} pointwise and having a small Morse index, whose definition is recalled below. By standard regularity theory, any solution $u \in \cal Y$ belongs to $ C^2(\Omega_{0\beta}\cup \Gamma_{\! 0} \cup \Gamma_{\! \beta})$. As mentioned in the Introduction, problem~\eqref{P} can be rewritten as~\eqref{instance}. Of course, if $u \in \cal Y$ then for every $R > r_1$ we have $u \in H^1(\Omega_{0\beta} \cap B_R)$, where $B_R \subset \R^N$ denotes the ball with radius~$R$ centered at the origin. Now we do not assume that $u_\theta$ belongs to $H^1(\Omega_{0\beta} \cap B_R)$, but we require $u_\theta$ to extend continuously up to the set $d(x) = r_1$ when $r_1 > 0$. Roughly speaking, the reason is that in order to derive a conclusion on~$u_\theta$ from a bound on the Morse index (see below) we have to approximate $u_\theta$ with some compactly supported function: hence $u_\theta$ must vanish on $\gamma_{0\beta}$ either in the sense of traces or as a continuous function. In this part of the paper we assume we are in the second case: the assumption is used, for instance, in the proof of~\eqref{convergence}. 
To proceed further, for any subset $D \subset \overline \Omega_{0\beta}$ we denote by $C^1_0(D)$ the set of functions $\varphi \in C^1(D)$ having a compact support contained in $D$.

\begin{definition}
We say that a solution $u$ to \eqref{P} 
\begin{itemize} 
\item[-] is stable, and has Morse index $m(u) = 0$, if $Q_{\Omega_{0\beta}}(\varphi)\geq 0$ for every  $\varphi \in C^1_{0}(\Omega_{0\beta}\cup \Gamma_{\! 0} \cup \Gamma_{\! \beta})$.
\item[-] has Morse index $m(u) = M\geq 1$ if $M$ is the maximal dimension of a vector subspace~$X \subset C^1_{0}(\Omega_{0\beta}\cup \Gamma_{\! 0} \cup \Gamma_{\! \beta})$ such that $Q_{\Omega_{0\beta}}(\varphi)< 0$ for every $\varphi
\in X\setminus\{0\}$.
\end{itemize}
\end{definition}
\noindent We will investigate solutions with $m(u) \le 1$. By definition, if $m(u) = 0$ then the bilinear form
\begin{equation}\label{bilinear}
\B(\varphi,\psi):=\int_{\Omega_{0\beta}} \Big( \nabla \varphi \, \nabla \psi-f'(d(x),z,u) \, \varphi \, \psi \Big) \, dx
\end{equation}
is positive semidefinite in $C^1_{0}(\Omega_{0\beta}\cup \Gamma_{\! 0} \cup \Gamma_{\! \beta}) \times C^1_{0}( \Omega_{0\beta}\cup \Gamma_{\! 0} \cup \Gamma_{\! \beta})$, and therefore the Cauchy-Schwarz inequality holds:
\begin{equation}\label{cs0}
0 \le \big( \B(\varphi,\psi) \big)^2
\leq
Q_{\Omega_{0\beta}}(\varphi) \, Q_{\Omega_{0\beta}}(\psi)
.
\end{equation}
Since $u$ is locally bounded, so is $f'(d(x),z,u)$ and the definition of $\B(\varphi,\psi)$ extends to $H^1_{\gamma_{0\beta} \cup \partial B_R}(D) \times H^1_{\gamma_{0\beta} \cup \partial B_R}(D)$, where $D = \Omega_{0\beta} \cap B_R$, $R > r_1$, and $H^1_{\gamma_{0\beta} \cup \partial B_R}(D)$ denotes the completion of $C^1_0(\left(\Omega_{0\beta}\cup \Gamma_0\cup\Gamma_\beta\right)\cap B_R)$ in the $H^1$-norm. Of course, positive semidefiniteness is preserved under such an extension.

\begin{remark}
The integral in~\eqref{bilinear} makes sense even if we replace $\varphi$ with some function $v$ having an unbounded support, because the only contribution comes from the bounded set $\mathop{\rm supp} \psi$: this will often be done in the sequel in order to prove that $\B(v,\psi) = 0$ for all $\psi \in C^\infty_0(D)$, i.e., $v$ is locally a weak solution of $L_u \, v = 0$. In such a case, however, $Q_{\Omega_{0\beta}}(v)$ is undefined, in general, and \eqref{cs0} makes no sense: this difficulty is overcome by means of a truncation technique inspired by~\cite{gladiali-pacella-weth} (see Section~\ref{properties} for details).
\end{remark}

\begin{remark}\label{stability}
If $m(u) = 1$ then there exists a one-dimensional subspace $X = \{\, \lambda\varphi_1 : \lambda \in \mathbb R \,\} \subset C^1_{0}(\Omega_{0\beta}\cup \Gamma_{\! 0} \cup \Gamma_{\! \beta})
$ such that $Q_{\Omega_{0\beta}}(\lambda\varphi_1)< 0$ for every $\lambda \ne 0$. In such a case, we have
$$
\mbox{$Q_{\Omega_{0\beta}}(\varphi) \ge 0$ for every $\varphi \in C^1_{0}(\Omega_{0\beta}\cup \Gamma_{\! 0} \cup \Gamma_{\! \beta})$ such that $\mathop{\rm supp} \varphi \cap \mathop{\rm supp} \varphi_1 = \emptyset$.}
$$
Indeed, if there were $\varphi_2 \in C^1_{0}(\Omega_{0\beta}\cup \Gamma_{\! 0} \cup \Gamma_{\! \beta})$ supported in  $\Omega_{0\beta}\cup \Gamma_{\! 0} \cup \Gamma_{\! \beta} \setminus \mathop{\rm supp} \varphi_1$
 and such that $Q_{\Omega_{0\beta}}(\varphi_2) < 0$, we would immediately construct the two dimensional subspace $W = \{\, a \varphi_1 + b \varphi_2 : a,b \in \mathbb R \,\} \subset C^1_{0}(\Omega_{0\beta}\cup \Gamma_{\! 0} \cup \Gamma_{\! \beta})$ such that $Q_{\Omega_{0\beta}}(\varphi) < 0$ for every $\varphi \in W \setminus \{0\}$, thus contradicting $m(u) = 1$. The argument extends to the more general case when $m(u) < +\infty$: see, for instance, \cite[Remark~2.7]{gladiali-pacella-weth}.

\end{remark}

The main result in this part of the paper is the extension of Theorem \ref{teo-1} to the case of the unbounded domain $\Omega_{0\beta}$ defined just before. In particular we will prove the following:
\begin{theorem}\label{teo-2}
Assume $u \in \cal Y$ is a solution to~\eqref{P} with Morse index $m(u)\leq 1$. Assume further $f(r,z,u)$ is convex with respect to~$u$.
\begin{enumerate}
\item If\/ $m(u) = 0$ then $u$ is constant with respect to the angular variable\/~$\theta$.
\item If\/ $m(u) = 1$ then $u$ is either independent of\/~$\theta$ or strictly monotone with respect to\/~$\theta$ in\/ $\Omega_{0\beta}$, in which case $u$ has a nonvanishing derivative~$u_\theta$ and $\inf _{\psi \in C^1_{0}(\Omega_{0\beta})}
Q_{\Omega_{0\beta}}(\psi)= 0$.
\end{enumerate}
\end{theorem}

The rest of the paper is devoted to prove Theorem \ref{teo-2}. First, in Section \ref{se:7} we introduce a suitable function space to which the solutions belong and we prove that some sequences of boundary integrals should converge to zero, which is useful 
to manage the boundary term coming from an integration by parts of the
function $u_\theta$ or the function $w_\alpha$ over a bounded subset
of $\Omega_{0\beta}$.
In Section \ref{se:8} we prove a sign preservation property, in Section \ref{properties} we extend the splitting lemma and in Section \ref{se:10} we conclude the proof. \\
We remark that the most difficult point to treat is the extension of Part VI of the proof of Theorem \ref{teo-1}. Indeed we cannot use the maximum principle in small domains and we need to introduce a suitable torus $T$, invariant under cylindrical rotations, to overcome this difficulty. 

\section{A suitable function space}\label{se:7}

As mentioned before, we consider (possibly unbounded) solutions to \eqref{P} belonging (in particular) to the set $\fspace = \{\, v \in C^2(\Omega_{0\beta}) \cap C^1(\Omega_{0\beta} \cup \Gamma_{\! 0} \cup \Gamma_{\! \beta}) \cap C^0(\overline \Omega_{0\beta}) : \nabla v \in L^2(\Omega_{0\beta},\mathbb R^N) \,\}$. We collect here some properties of the space $\fspace$ and its subspaces $\fspace_\alpha$, where for $\alpha \in (0,\beta]$, we denote by~$\fspace_\alpha$ the subset of all $v \in \fspace$ such that $v = 0$ on~$\Gamma_{\!\alpha}$. Note that the function $w_\frac\beta2$ associated to a solution $u \in \fspace$ obviously belongs to~$\fspace_\frac\beta2$. For $\alpha \in (0,\beta]$ we let  $\Sigma_{\alpha R}$ be the surface $\Omega_{0\alpha} \cap \partial B_R$, where $R > r_1$. By Tonelli's theorem, for every $v \in E$ and almost every $R > r_1$ we have
$$
 \int_{\Sigma_{\beta R}}
|\nabla v|^2 \, d\Sigma
<
+\infty
.
$$
Of course, if $N = 2$ then $\Sigma_{\beta R}$ reduces to an arc, 
$v \in C^1(\overline \Sigma_{\beta R})$ and the integral above is finite for every $R > r_1$. 
\begin{lemma}
Let $\alpha \in (0,\beta]$, $\Sigma_{\alpha R}=\Omega_{0\alpha} \cap \partial B_R$ and $\nu $ be the outward normal to the sphere~$\partial B_R$. For every $v \in \fspace_\alpha$ and almost every $R > r_1$ the following inequalities hold:
\begin{align}
\label{estimate}
\sqrt{\lambda_1(r_1,R) \,}
\,
\left|
\int_{ \Sigma_{\alpha R}}
v \, \frac{\, \partial v \,}{\partial \nu}
\, d\Sigma
\right|
&\le
\int_{ \Sigma_{\alpha R}}
|\nabla v|^2 \, d\Sigma
,
\\
\noalign{\medskip}
\label{estimate2}
\lambda_1(r_1,R)
\,
\int_{\Sigma_{\alpha R}} v^2 \, d\Sigma
&\le
\int_{ \Sigma_{\alpha R}}
|\nabla v|^2 \, d\Sigma
,
\end{align}
where $\lambda_1(r_1,R) > 0$ is given by
$$
\lambda_1(r_1,R)
=
\inf_{\genfrac{}{}{0pt}{1}{w \in H^1( \Sigma_{\alpha R}) \setminus \{0\}}{w = 0 \mbox{ \scriptsize on } \Gamma_{\!\alpha} \cap \partial B_R}}
\frac{\, \int_{\Sigma_{\alpha R}} |\nabla_{\!\tau\,} w|^2 \, d\Sigma \,}
{\int_{ \Sigma_{\alpha R}} w^2 \, d\Sigma}
$$
and\/ $\nabla_{\!\tau\,} w$ denotes the intrinsic gradient of~$w$ over the surface $\Sigma_{\alpha R}$.
\end{lemma}

\begin{proof}
By the Cauchy-Schwarz inequality, for almost every $R > r_1$ we may write
\begin{equation}\label{Cauchy-Schwarz}
\left|
\int_{ \Sigma_{\alpha R}}
v \, \frac{\, \partial v \,}{\partial \nu}
\, d\Sigma
\right|
\le
\left(
\int_{ \Sigma_{\alpha R}} v^2 \, d\Sigma
\right)^{\!\frac12}
\left(
\int_{ \Sigma_{\alpha R}}
\Big(\frac{\, \partial v \,}{\, \partial \nu \,}\Big)^{\! 2} \, d\Sigma
\right)^{\!\frac12}
\!\!.
\end{equation}
By Poincar\'e's inequality, and recalling that $\nabla v = \nabla_{\!\tau\,} v + \frac{\, \partial v \,}{\, \partial \nu \,} \, \nu$, the first integral in the right-hand side is estimated as follows:
\begin{align*}
\lambda_1(r_1,R)
\,
\int_{ \Sigma_{\alpha R}} v^2 \, d\Sigma
&\le
\int_{ \Sigma_{\alpha R}}
|\nabla_{\!\tau\,} v|^2 \, d\Sigma
\\
\noalign{\medskip}
&\le
\int_{ \Sigma_{\alpha R}}
|\nabla v|^2 \, d\Sigma
,
\end{align*}
which yields~\eqref{estimate2}. Furthermore, we obviously have
$$
\int_{ \Sigma_{\alpha R}}
\Big(\frac{\, \partial v \,}{\, \partial \nu \,}\Big)^{\! 2} \, d\Sigma
\le
\int_{\Sigma_{\alpha R}}
|\nabla v|^2 \, d\Sigma
$$
and~\eqref{estimate} follows by multiplying the last two estimates.
\end{proof}

Inequalities~\eqref{estimate}-\eqref{estimate2} enter in the proof of the following result:

\begin{lemma}\label{vanishes}
Let the domain\/ $\Omega_{0\beta}$ and the function spaces $\fspace$ and $\fspace_\alpha$ be defined as above.
\begin{enumerate}
\item\label{C1} For every $v \in \fspace_\alpha$ and $R_0 > r_1$, the iterated integral
$$
\int_{R_0}^{+\infty}
\bigg(
\frac1{\, R^2 \,}
\int_{ \Sigma_{\alpha R}}
v^2 \, d\Sigma
\bigg)
dR
$$
converges to a finite value, hence it is infinitesimal as $R_0 \to +\infty$.
\item For every $v \in \fspace_\alpha$ there exists a sequence $R_k \to +\infty$ such that
\begin{equation}\label{limit}
\lim_{k \to +\infty}
\int_{ \Sigma_{\alpha R_k}}
v \, \frac{\, \partial v \,}{\partial \nu}
\, d\Sigma
=
0.
\end{equation}
\item For every $v \in \fspace$ such that $v_\theta \in C^1(\overline \Omega_{0\alpha})$ there exists a sequence $R_k \to +\infty$ such that the negative part $v_\theta^- = \min\{\, 0,v_\theta \,\} \le 0$ satisfies
\begin{equation}\label{v_theta}
\lim_{k \to +\infty}
\int_{ \Sigma_{\alpha R_k}}
v_\theta^- \, \frac{\, \partial v_\theta^- \,}{\partial \nu}
\, d\Sigma
=
0.
\end{equation}
\end{enumerate}
\end{lemma}

\begin{proof}
1. By the scaling property of the eigenvalues we may write
\begin{equation}\label{asymptotics}
\Big( \frac{R}{\, R_0 \,} \Big)^{\! 2} \, \lambda_1(r_1,R)
=
\lambda_1 \Big( \frac{\, R_0 \,}R \, r_1, \, R_0 \Big)
\stackrel{R \to +\infty}{\longrightarrow}
\lambda_1(0, R_0) > 0
.
\end{equation}
We note in passing that recent advances on spectral asymptotics are found, for instance, in~\cite{Jimbo} and in the references cited there. From~\eqref{asymptotics} it follows that there exists $\varepsilon_0 > 0$ such that $\lambda_1(r_1,R) \ge \varepsilon_0/R^2$ for $R$ large. This and~\eqref{estimate2} imply
\begin{align*}
\frac{\, \varepsilon_0 \,}{\, R^2 \,}
\int_{ \Sigma_{\alpha R}}
v^2 \, d\Sigma
\le
\int_{ \Sigma_{\alpha R}}
|\nabla v|^2 \, d\Sigma
.
\end{align*}
Taking into account that the right-hand side is summable by assumption over the interval $(r_1,+\infty)$, Claim~\ref{C1} follows.

2. If the claim were false, then by~\eqref{estimate} there would be $R_0 > r_1$ and $\varepsilon_0 > 0$ such that for almost every $R \ge R_0$
$$
\int_{\Sigma_{\alpha R}}
|\nabla v|^2 \, d\Sigma
\ge \varepsilon_0 \sqrt{\lambda_1(r_1,R) \,}
.
$$
Taking the asymptotics~\eqref{asymptotics} into account, for almost every $R$ large enough we may write
$$
\int_{ \Sigma_{\alpha R}}
|\nabla v|^2 \, d\Sigma
\ge
\frac{\, \varepsilon_0 \,}2
\,
\frac{\, R_0 \,}R
\,
\sqrt{\lambda_1(0,R_0) \,}
.
$$
Integrating both sides in $dR$ we would get $\Vert \nabla v \Vert_{L^2(\Omega_{0\alpha})} = +\infty$, which is in contrast with the definition of the function space $\fspace_\alpha$.

\smallskip3. Since $\nabla v \in L^2(\Omega_{0\beta})$, the function
$$
f(R)
=
\int_{\Omega_{0\alpha} \cap B_R} (v_\theta^-)^2 \, dx
$$
has a finite limit for $R \to +\infty$. By differentiation we find
$$
f'(R)
=
\int_{ \Sigma_{\alpha R}}
(v_\theta^-)^2 \, d\Sigma.
$$
Differentiating once more yields
$$
f''(R)
=
\int_{\Sigma_{\alpha R}}
2v_\theta^- \, \frac{\, \partial v_\theta^- \,}{\partial \nu}
\, d\Sigma
+
\frac{\, N - 1 \,}R \, f'(R),
$$
which may be rewritten as
$$
\Big(
\frac{\, f'(R) \,}{\, R^{N - 1} \,}
\Big)'
=
\frac1{\, R^{N - 1} \,}
\int_{ \Sigma_{\alpha R}}
2v_\theta^- \, \frac{\, \partial v_\theta^- \,}{\partial \nu}
\, d\Sigma
.
$$
If the last integral were larger than $\varepsilon_0 > 0$ for $R \ge R_0$, then, integrating both sides over the interval $(R_0,R)$, we would get
$$
\frac{\, f'(R) \,}{\, R^{N - 1} \,}
=
C
+
\varepsilon_0
\int_{R_0}^R \frac{dr}{\, r^{N - 1} \,}
$$
which implies $f'(R) \to +\infty$ as $R \to +\infty$. But this is impossible because $f(R)$ converges to a finite limit. A similar contradiction arises if we assume that the integral in~\eqref{v_theta} keeps smaller than $-\varepsilon_0 < 0$.
\end{proof}

We note in passing that~\eqref{limit} continues to hold in the special case when $N = 2$ even though $v$ does not vanish along $\Gamma_{\!\alpha}$. However, the proof is different:

\begin{proposition}
In the special case when $N = 2$, for every $v \in \fspace$ and every $\alpha \in (0,\beta]$ there exists a sequence $R_k \to +\infty$ such that\/~{\rm(\ref{limit})} holds.
\end{proposition}

\begin{proof}
By the Cauchy-Schwarz inequality we may still write~\eqref{Cauchy-Schwarz}, which holds now for every $R > r_1$ because  $v \in C^1(\overline \Omega_{0\beta} \setminus B_R)$. However, the Poincar\'e inequality fails because of the lack of boundary conditions. Nevertheless, the first integral in the right-hand side is estimated as follows. Fix $r_0 > r_1$, so that the function $v(r_0, \theta)$ of the variable~$\theta$ belongs to $C^1([0,\alpha])$. By the fundamental theorem of calculus we may write
\begin{align*}
|v(R,\theta) - v(r_0,\theta)|
&\le
\int_{r_0}^R
\Big| \frac{\, \partial v \,}{\, \partial r \,} \Big| \, dr
\\
\noalign{\medskip}
&\le
\left(
\int_{r_0}^R
\frac1{\, r \,} \, dr
\right)^{\!\frac12}
\left(
\int_{r_0}^R
r \, \Big( \frac{\, \partial v \,}{\, \partial r \,} \Big)^{\! 2} \, dr
\right)^{\!\frac12}
\!\!.
\end{align*}
Recall that the area element in $\R^2$ is $r \, dr \, d\theta$. Therefore, since $\nabla v \in L^2(\Omega_{0\alpha})$, the last integral converges almost everywhere as $R \to +\infty$ to some function $\psi(\theta) < +\infty$ which is summable over the interval $(0,\alpha)$. Hence $v^2(R,\theta) \le C \, \psi(\theta) \log R$ for almost every~$\theta$ and for large $R$, with a convenient constant~$C$. By plugging this estimate into~\eqref{Cauchy-Schwarz} we obtain
$$
\left|
\int_{ \Sigma_{\alpha R}}
v \, \frac{\, \partial v \,}{\partial \nu}
\, d\Sigma
\right|
\le
C \, (R \, \log R)^\frac12
\left(
\int_{ \Sigma_{\alpha R}}
|\nabla v|^2 \, d\Sigma
\right)^{\!\frac12}
$$
where we have redefined the constant~$C$. If the claim were false, there would be $R_0 > r_0$ and $\varepsilon_0 > 0$ such that for every $R \ge R_0$
$$
\int_{ \Sigma_{\alpha R}}
|\nabla v|^2 \, d\Sigma
\ge
\frac{\varepsilon_0}{\, R \, \log R \,}
.
$$
Integrating both sides over the interval $[R_0,+\infty)$ we would get $\Vert \nabla v \Vert_{L^2(\Omega_{0\alpha})} = +\infty$, which is in contrast with the definition of the function space $\fspace$.
\end{proof}

\begin{remark}
In general, \eqref{limit} fails for $v \in \fspace$ and $N \ge 3$. For instance, we may take $r_1 > 0$, $\varepsilon \in (0, \, \frac{N}2 - 1)$ and
$$
v(x) = 1 - \frac1{\, |x|^{\frac{N}2 - 1 + \varepsilon} \,},\
\mbox{hence}\
\nabla v(x)
=
\frac{\, \frac{N}2 - 1 + \varepsilon \,}
{R^{\frac{N}2 + \varepsilon}}
\,
\nu\
\mbox{when $x \in \Sigma_{\beta R}$.}
$$
In this example we have $v \in \fspace$ because $\varepsilon > 0$. However,
$$
\int_{ \Sigma_{\beta R}}
v \, \frac{\, \partial v \,}{\partial \nu}
\, d\Sigma
=
\Big( 1 - \frac1{\, R^{\frac{N}2 - 1 + \varepsilon} \,} \Big)
\,
\Big( \frac{N}2 - 1 + \varepsilon \Big)
\,
R^{\frac{N}2 - 1 - \varepsilon}
\,
| \Omega_{0\beta} \cap \partial B_1|
\to +\infty
.
$$
\end{remark}

\begin{remark}
A function $v \in \fspace_\alpha \subset \fspace$ may well be unbounded. To see this, recall that for every $N \ge 2$ there exists a function $\varphi \in C^\infty_0(B_1)$ such that $\Vert \nabla \varphi \Vert_{L^2(B_1)}$ is arbitrarily small and $\sup |\varphi|$ is arbitrarily large. Then we may take a sequence of centers $x_k \in \Omega_{0\alpha}$ and a sequence of radii $r_k > 0$ such that $|x_k| \to +\infty$ and the balls $B_k = B(x_k,r_k) \subset \Omega_{0\alpha}$ are pairwise disjoint. For each~$k$ we fix $\varphi_k \in C^\infty_0(B_k)$ such that $\Vert \nabla \varphi_k \Vert^2_{L^2(B_k)} \le 2^{-k}$ and $\sup |\varphi_k| \ge k$. The series
$$
v(x) = \sum_{k = 1}^{+\infty} \varphi_k(x)
$$
trivially converges in $\fspace_\alpha$ because for every $x$ it has at most one nonzero term, and it is apparent that $v$ is unbounded and satisfies
$$
\Vert \nabla v \Vert^2_{L^2(\Omega_{0\alpha})}
\le
\sum_{k = 1}^{+\infty} 2^{-k}
=
1
.
$$
\end{remark}

\begin{remark}
If $N=2$ there exist radial, unbounded functions in the space~$\fspace$. For instance, we may take $v(x) = \log \log |x|$ for $|x|>1$: in this case $\nabla v = (|x| \, \log |x|)^{-1} \, \frac {x}{|x|}$ and therefore $\nabla v \in L^2(\Omega_{0\beta})$ provided that $r_1 > 1$. If, instead, $N \ge 3$, every radial function $v \in \fspace$ is bounded. To see this, let us write $v(x) = f(|x|)$. In the special case when $r_1 = 0$, the surface integral
$$
I(R)
=
\int_{ \Sigma_{\beta R}}
|\nabla v|^2 \, d\Sigma
$$
equals $\frac\beta{2\pi} \, N \omega_N \, \psi(R)$, where $\psi(R) = R^{N - 1} \, (f'(R))^2$ and $N \omega_N = |\partial B_1|$ is the Hausdorff measure of the unit spherical surface in~$\mathbb R^N$. Hence, $\psi(R)$ must be summable over the interval $(0,+\infty)$ because $\nabla v \in L^2(\Omega_{0\beta})$. By the Cauchy-Schwarz inequality, for every $R > 1$ we may write
\begin{align*}
|f(R) - f(1)|
=
\left| \int_1^R f'(R) \, dR \right|
\le
\int_1^R |f'(R)| \, dR
&=
\int_1^R
\Big(\frac{\, \psi(R) \,}{\, R^{N - 1} \,}\Big)^{\! \frac12} \, dR
\\
&\le
\left(
\int_1^R
\psi(R) \, dR
\right)^{\!\! \frac12}
\left(
\int_1^R
\frac1{\, R^{N - 1} \,}
\right)^{\!\! \frac12}
,
\end{align*}
which shows that the right-hand side keeps bounded as $R \to +\infty$. In the general case when $r_1 \ge 0$ the integral $I(R)$ is asymptotic to $\psi(R)$, and the conclusion follows similarly.
\end{remark}

\section{Sign preservation in unbounded domains}\label{se:8}

In this section we prepare two fundamental lemmas which replace Lemma~\ref{sign} in the case when the domain is unbounded. In order to prove the first result, we need to turn the negative part $v^-$ of a supersolution of the linear equation $L_u \, v = 0$ into a function with bounded support, thus allowing to use the Cauchy-Schwarz inequality~\eqref{cs0}. Following~\cite{gladiali-pacella-weth}, this is accomplished by a family of cut-off functions $\xi_R\in C^\infty_0(\R^N)$, $R > 0$, such that
\begin{equation}\label{gradient}
\mbox{$0\leq \xi_R(x)\leq 1$ and $|\nabla \xi_R(x)|<\displaystyle\frac2{\, R\ }$ for every $x \in \mathbb R^N$,}
\quad
\xi_R(x)=\begin{cases}
1, & \text{ if }|x|<R,
\\
0, & \text{ if }|x|>2R.
\end{cases}
\end{equation}
Later on we will need a further family of cut-off functions in order to deal with possible singularities at $d(x) = r_1$, i.e., the case when $u \not \in W^{2,2}(\Omega_{0\beta} \cap B_R)$, $R > r_1$.
Hence we denote by $\zeta_\delta(x) \in C^\infty_0(\mathbb R^N)$, $\delta > 0$, a function depending only on~$d(x)$ and such that
\begin{equation}\label{zeta}
\mbox{$0\leq \zeta_\delta(x)\leq 1$ and $|\nabla \zeta_\delta(x)| < \displaystyle\frac3{\, \delta \,}$ for all $x \in \mathbb R^N$,}
\quad
\zeta_\delta(x)=\begin{cases}
0, & \text{ if $d(x) < r_1 + \delta/2$,}
\\
1, & \text{ if $d(x) > r_1 + \delta$.}
\end{cases}
\end{equation}

\begin{lemma}[Sign preservation in unbounded domains]\label{sign-preservation}
Fix\/ $\alpha \in (0,\beta)$ and let $K$ be a (possibly empty) compact subset of\/~$\R^N$. Suppose that the set difference $D_\alpha = \Omega_{0\alpha} \setminus K$ is connected, and let $v \in E_\alpha$ satisfy $L_u \, v \ge 0$ in $D_\alpha$ for some $u \in L^\infty(\Omega_{0\alpha})$. Define\/ $\tilde \Gamma_{\! 0} = \partial D_\alpha \cap \Gamma_{\! 0}$, and assume
\begin{equation}\label{stable}
\inf_{\varphi \in
C^1_{0}(D_\alpha \cup \tilde \Gamma_{\! 0})}
Q_{\Omega_{0\alpha}}(\varphi)
\ge
0
.
\end{equation}
Assume, further, that $v \ge 0$ on $\partial D_\alpha \setminus \Gamma_{\! 0}$ together with $v \ge 0$ or\/ $\partial v / \partial \nu \ge 0$ on\/ $\tilde \Gamma_{\! 0}$. If\/ $v$ does not vanish identically in~$D_\alpha$, then either $v > 0$ in~$D_\alpha$ or $v < 0$ there.
\end{lemma}

\begin{proof}
Let $r_K = \max_{x \in K} d(x)$ and define $R_0 = \max\{\, r_1, r_K \,\}$. By~\eqref{stable}, for every $R > R_0$ the bilinear form~\eqref{bilinear} is positive semidefinite in the Cartesian product $H^1_{\partial D \setminus \Gamma_0}(D) \times H^1_{\partial D \setminus \Gamma_0}(D)$, where $D = D_\alpha \cap B_{2R}$, and the Cauchy-Schwarz inequality~\eqref{cs0} holds. Let us check that the negative part $v^-(x) = \min \{\ 0, \, v(x)  \,\} \le 0$ is a weak solution of
\begin{equation}\label{local}
L_u \, v^- = 0,
\end{equation}
i.e., $\B(v^-, \, \psi) = 0$ for every $\psi \in C^\infty_0(D_\alpha)$. By making $R$ larger we may achieve that $\mathop{\rm supp} \psi \subset B_R$, hence $\psi \in C^\infty_0(D)$. Letting $\varphi(x) = \xi_R(x) \, v(x) \in H^1_{\partial D \setminus \Gamma_0}(D)$ in~\eqref{cs0} we obtain
\begin{equation}\label{cs}
0
\le
\big( \B(v^-, \psi) \big)^2
=
\big( \B(\xi_R \, v^-, \psi) \big)^2
\leq
Q_{\Omega_{0\alpha}}(\xi_R \, v^-) \, Q_{\Omega_{0\alpha}}(\psi)
.
\end{equation}
The value of $Q_{\Omega_{0\alpha}}(\xi_R \, v^-)$ may depend on~$R$: however we claim that it is infinitesimal as $R \to +\infty$. Indeed, by differentiation we find
\begin{align}
\nonumber
|\nabla \xi_R|^2 \, \varphi^2
+
\nabla(\xi_R^2 \, \varphi) \, \nabla \varphi
&=
|\nabla \xi_R|^2 \, \varphi^2
+
2\xi_R \, \varphi \, \nabla \xi_R \, \nabla \varphi + \xi_R^2 \, |\nabla \varphi|^2
\\
\noalign{\medskip}
\label{calculus}
&=
|\nabla(\xi_R \, \varphi)|^2
,
\end{align}
an identity that will be repeatedly used in the sequel, with different choices of~$\varphi$. Taking into account the boundary conditions $v^- = 0$ on $\partial D \setminus \Gamma_{\! 0}$ and $v^- \, \frac{\partial v^-}{\partial \nu} \le 0$ on $\partial D \cap \Gamma_{\! 0}$, multiplication of the inequality $L_u \, v \ge 0$ by $\xi_R^2 \, v^-$ and integration over~$D$ yields
$$
\int_D
\Big\{
\nabla(\xi_R^2 \, v^-) \, \nabla v^-
-
f'(d(x),z,u)
\,
\xi_R^2 \, (v^-)^2
\Big\}
\, dx
\le 0
.
$$
Letting $\varphi = v^-$ in~\eqref{calculus}, and using~\eqref{gradient}, this is turned into
\begin{equation}\label{parts}
0 \le
Q_{\Omega_{0\alpha}}(\xi_R \, v^-)
\le
\int_D
|\nabla \xi_R|^2 \, (v^-)^2 \, dx
\le
\frac4{\, R^2 \,}
\int_R^{2R}
\bigg(
\int_{ \Sigma_{\alpha r}}
(v^-)^2 \, d\Sigma
\bigg) \, dr
,
\end{equation}
where $ \Sigma_{\alpha r} = \Omega_{0\alpha} \cap \partial B_r$. To manage with the last integral, note that
\begin{align*}
\frac1{\, (2R)^2 \,}
\int_R^{2R}
\bigg(
\int_{ \Sigma_{\alpha r}}
(v^-)^2 \, d\Sigma
\bigg)
\, dr
&\le
\int_R^{2R}
\bigg(
\frac1{\, r^2 \,}
\int_{ \Sigma_{\alpha r}}
(v^-)^2 \, d\Sigma
\bigg)
\, dr
\\
\noalign{\medskip}
&\le
\int_R^{+\infty}
\bigg(
\frac1{\, r^2 \,}
\int_{\Sigma_{\alpha r}}
(v^-)^2 \, d\Sigma
\bigg)
\, dr
,
\end{align*}
therefore by Claim~\ref{C1} of Lemma~\ref{vanishes} the right-hand side of~\eqref{parts} is infinitesimal. But then \eqref{cs} implies $\B(v^-, \psi) = 0$. Since $\psi$ is arbitrary, $v^-$ is a weak solution of equation~\eqref{local}, as claimed. Now the unique continuation property implies that either $v^- \equiv 0$ or $v^- < 0$ in~$D_\alpha$. In the last case we obviously have $v < 0$ in $D_\alpha$. If, instead, $v^- \equiv 0$ in $D_\alpha$, then $v \ge 0$, and since $L_u \, v \ge 0$ in $D_\alpha$, the strong maximum principle implies that either $v \equiv 0$ or $v > 0$ in the connected set~$D_\alpha$.
\end{proof}

\begin{remark}
If we restrict to the case when $v \ge 0$ on~$\Gamma_{\! 0}$, we may replace assumption~\eqref{stable} with the weaker requirement
$$
\inf_{\varphi \in C^1_{0}(D_\alpha)}
Q_{\Omega_{0\alpha}}(\varphi)
\ge
0
.
$$
\end{remark}

In the sequel we will need to apply a similar argument to $u_\theta$ and $u^-_\theta$, using a double truncation. Therefore we prepare:

\begin{proposition}\label{small}
Assume $u \in \cal Y$ is a solution to~{\rm(\ref{P})}, and define $u^-_\theta(x) = \min\{\, u_\theta(x), \, 0 \,\}$. Let\/ $\xi_R, \zeta_\delta$ be the cut-off functions in~{\rm(\ref{gradient})}-{\rm(\ref{zeta})}. For every $\varepsilon > 0$ there exist\/ $\delta > 0$ and $R > r_1$ such that\/ $Q_{\Omega_{0\beta}}(\zeta_\delta \, \xi_R \, u_\theta) \in [0,\varepsilon)$. If, furthermore, $u^-_\theta = 0$ on~$\Gamma_{\! \frac\beta2}$ then we may also achieve\/ $Q_{\Omega_{0\frac\beta2}}(\zeta_\delta \, \xi_R \, u^-_\theta) \in [0,\varepsilon)$.
\end{proposition}

\begin{proof}
First of all, we obtain
\begin{equation}\label{R}
\int_{\Omega_{0\beta}}
|\nabla \xi_R|^2 \, u_\theta^2 \, dx
<
\frac\varepsilon{\, 2 \,}
\end{equation}
by taking $R$ large. Indeed, we have
$$
\int_{\Omega_{0\beta}}
|\nabla \xi_R|^2 \, u_\theta^2 \, dx
\le
\frac4{\, R^2 \,}
\int_R^{2R}
\bigg(
\int_{ \Sigma_{\beta r}}
u_\theta^2 \, d\Sigma
\bigg) \, dr
,
$$
where $\Sigma_{\beta r}= \Omega_{0\beta} \cap \partial B_r$. Since $|u_\theta(x)| \le d(x) \, |\nabla u(x)| \le r \, |\nabla u|$ on $\partial B_r$, we may write
\begin{align*}
\frac1{\, 4R^2 \,}
\int_R^{2R}
\bigg(
\int_{ \Sigma_{\beta r}}
u_\theta^2 \, d\Sigma
\bigg) \, dr
&\le
\frac1{\, 4R^2 \,}
\int_R^{2R}
r^2
\,
\bigg(
\int_{ \Sigma_{\beta r}}
|\nabla u|^2 \, d\Sigma
\bigg) \, dr
\\
\noalign{\medskip}
&\le
\int_R^{+\infty}
\bigg(
\int_{ \Sigma_{\beta r}}
|\nabla u|^2 \, d\Sigma
\bigg) \, dr
,
\end{align*}
where the right-hand side is infinitesimal as $R \to +\infty$. Hence we may choose $R$ so that~\eqref{R} holds. Next we observe that the function $u_\theta$, which is uniformly continuous in the compact set $\overline \Omega_{0\beta} \cap \overline B_{2R}$, converges uniformly to zero when $d(x) \to r_1$. Keeping this in mind, and letting $D_{\delta R} = \{\, x \in \Omega_{0\beta} : \zeta_\delta(x) \, \xi_R(x) > 0 \,\}$, we may check that
\begin{equation}\label{convergence}
\lim_{\delta \to 0}
\int_{D_{\delta R}}
|\nabla (\zeta_\delta \, \xi_R)|^2 \, u_\theta^2 \, dx
=
\int_{\Omega_{0\beta}}
|\nabla \xi_R|^2 \, u_\theta^2 \, dx
.
\end{equation}
To see this, choose $\mu > 0$ and let $\delta$ be so small that $u_\theta^2 < \mu$ in the set $S_{\delta R} = \{\, x \in \Omega_{0\beta} \cap B_{2R} : d(x) < r_1 + \delta \,\}$. Since $\nabla(\zeta_\delta \, \xi_R) = \zeta_\delta \, \nabla \xi_R + \xi_R \, \nabla \zeta_\delta$, we have
\begin{align*}
\Big|
|\nabla (\zeta_\delta \, \xi_R)|^2
-
|\nabla \xi_R|^2
\Big|
&=
\Big|
(\zeta_\delta^2 - 1) \, |\nabla \xi_R|^2
+
2 \, \zeta_\delta \, \xi_R \, \nabla \zeta_\delta \, \nabla \xi_R
+
\xi_R^2 \, |\nabla \zeta_\delta|^2
\Big|
\\
\noalign{\medskip}
&\le
2 \, |\nabla \xi_R|^2
+
2 \, |\nabla \zeta_\delta|^2
\le
\frac8{\, R^2 \,} + \frac{\, 18 \,}{\, \delta^2 \,}
\end{align*}
and the difference of the two preceding integrals is estimated as follows:
\begin{align*}
\left|
\int_{D_{\delta R}}
|\nabla (\zeta_\delta \, \xi_R)|^2 \, u_\theta^2 \, dx
-
\int_{\Omega_{0\beta}}
|\nabla \xi_R|^2 \, u_\theta^2 \, dx
\right|
\le
\mu
\,
|S_{\delta R}|
\,
\Big(
\frac8{\, R^2 \,} + \frac{\, 18 \,}{\, \delta^2 \,}
\Big)
.
\end{align*}
Since $R$ is kept fixed, we have $|S_{\delta R}| \le C \, \delta^2$ for some constant~$C$, hence
$$
\left|
\int_{D_{\delta R}}
|\nabla (\zeta_\delta \, \xi_R)|^2 \, u_\theta^2 \, dx
-
\int_{\Omega_{0\beta}}
|\nabla \xi_R|^2 \, u_\theta^2 \, dx
\right|
\le
\mu
\,
C
\,
\Big(
\frac{\, 8 \,\delta^2 \,}{\, R^2 \,} + 18
\Big)
.
$$
Finally, since $\mu$ is arbitrary,~\eqref{convergence} follows. Recalling~\eqref{R}, we conclude that
\begin{equation}\label{2epsilon}
\int_{D_{\delta R}}
|\nabla (\zeta_\delta \, \xi_R)|^2 \, u_\theta^2 \, dx
<
\varepsilon
\end{equation}
for convenient $R$ and~$\delta$. Now, multiplying the equation $L_u \, u_\theta = 0$ (see Proposition~\ref{solution}) by $(\zeta_\delta \, \xi_R)^2 \, u_\theta$ and integrating by parts over $\Omega_{0\beta} \cap B_{2R}$ we obtain
$$
\int_{D_{\delta R}}
\Big\{
\nabla \big((\zeta_\delta \, \xi_R)^2 \, u_\theta \big) \, \nabla u_\theta
-
(\zeta_\delta \, \xi_R)^2 \, u_\theta^2
\,
f'(d(x),z,u)
\Big\}
\, dx
= 0
.
$$
Replacing $\xi_R$ with $\zeta_\delta \, \xi_R$ in~\eqref{calculus} and letting $\varphi = u_\theta$ (see also~\eqref{parts}) this is turned into
\begin{equation}\label{therefore}
Q_{\Omega_{0\beta}}(\zeta_\delta \, \xi_R \, u_\theta)
=
\int_{D_{\delta R}}
|\nabla (\zeta_\delta \, \xi_R)|^2 \, u_\theta^2 \, dx
.
\end{equation}
Hence $Q_{\Omega_{0\beta}}(\zeta_\delta \, \xi_R \, u_\theta) \in [0,\varepsilon)$ for a convenient choice of~$\delta$. To prove the second claim it suffices to multiply the equation $L_u \, u_\theta = 0$ by $(\zeta_\delta \, \xi_R)^2 \, u_\theta^-$ and integrate by parts over $\Omega_{0\frac\beta2} \cap B_{2R}$. Since $u^-_\theta = 0$ on~$\Gamma_{\! \frac\beta2}$ by assumption, we obtain
$$
\int_D
\Big\{
\nabla\big((\zeta_\delta \, \xi_R)^2 \, u_\theta^-\big) \, \nabla u_\theta^-
-
(\zeta_\delta \, \xi_R)^2 \, (u_\theta^-)^2
\,
f'(d(x), z, u)
\Big\}
\, dx
= 0
,
$$
and therefore (see~\eqref{therefore})
$$
Q_{\Omega_{0\frac\beta2}}(\zeta_\delta \, \xi_R \, u_\theta^-)
=
\int_{D'_{\delta R}}
|\nabla(\zeta_\delta \, \xi_R)|^2 \, (u_\theta^-)^2 \, dx
\le
\int_{D_{\delta R}}
|\nabla(\zeta_\delta \, \xi_R)|^2 \, u_\theta^2 \, dx
,
$$
where  $D'_{\delta R} = \{\, x \in \Omega_{0\frac\beta2} : \zeta_\delta(x) \, \xi_R(x) > 0 \,\}$. The last integral is made arbitrarily small by a convenient choice of $\delta$ and $R$ (see~\eqref{2epsilon}), and the proof is complete.
\end{proof}

\section{Preliminaries in the unbounded half-sector}\label{properties}

The counterpart of Corollary~\ref{corollary} in an unbounded domain is the following

\begin{lemma}[Splitting lemma for unbounded domains]\label{lem-4.1}
Assume $u \in \cal Y$ is a solution to \eqref{P} with Morse index $m(u)\leq 1$, and let $\alpha \in (0,\beta)$. Then either\
\begin{equation}\label{4.1}
\inf_{\varphi\in C^1_{0}(\Omega_{0\alpha}\cup\Gamma_{\! 0})}
Q_{\Omega_{0\alpha}}(\varphi)\geq0
\end{equation}
or
\[\inf_{\varphi\in C^1_{0}(\Omega_{\alpha\beta}\cup\Gamma_{\! \beta})}Q_{\Omega_{\alpha\beta}}(\varphi)\geq0.\]
\end{lemma}

\begin{proof}
If the thesis is not true there exist two functions $\varphi_{0\alpha},\varphi_{\alpha\beta}$ supported in $\Omega_{0\alpha}\cup\Gamma_{\!0}$ and $\Omega_{\alpha\beta}\cup\Gamma_{\!\beta}$, respectively, such that $Q_{\Omega_{0\alpha}}(\varphi_{0\alpha}), Q_{\Omega_{\alpha\beta}}(\varphi_{\alpha\beta})<0$. Since they both vanish in a neighborhood of $\Gamma_{\!\alpha}$ we may extend them to the whole of $\Omega_{0\beta}$ as in Lemma \ref{lem-1}. The extended functions belong to $C^1_{0}(\Omega_{0\beta}\cup\Gamma_{\! 0} \cup \Gamma_{\! \beta})$ and are linearly independent. Then, as in the proof of Lemma \ref{lem-1} it is easy to prove that $Q_{\Omega_{0\beta}}(\phi)<0$ for every $\phi= a \, \varphi_{0\alpha} + b \, \varphi_{\alpha\beta}$ with $(a,b)\neq(0,0)$ contradicting the fact that the Morse index of $u$ in $\Omega_{0\beta}$ is at most one.
\end{proof}

Next we state a sign preservation property for $w_\frac\beta2$ which extends Corollary~\ref{cor-3} to unbounded domains.

\begin{lemma}[Sign preservation in the unbounded half-sector]\label{lem-4.2}
Assume $u \in \cal Y$ is a solution to~\eqref{P} with Morse index $m(u)\leq 1$. Assume further $f(r,z,u)$ is convex with respect to~$u$. If\/ $w_{\frac \beta 2}$ does not vanish identically in~$\Omega_{0\frac\beta2}$, then either\/ $w_{\frac \beta 2} > 0$ in~$\Omega_{0\frac\beta2}$ or $w_{\frac \beta 2} < 0$ there.
\end{lemma}

\begin{proof}
Letting $\alpha = \frac\beta2$ in Lemma \ref{lem-4.1}, we may suppose
\begin{equation}\label{4.1b}
\inf_{\varphi\in C^1_{0}(\Omega_{0\frac\beta2}\cup \Gamma_{\! 0})}
Q_{\Omega_{0\frac\beta2}}(\varphi)\geq0
.
\end{equation}
Thus, assumption~\eqref{stable} is satisfied, and the claim follows from Lemma~\ref{sign-preservation} with $K = \emptyset$, since, by the convexity of $f$, $w_{\frac \beta 2}$ satisfies $L_u w_{\frac \beta 2}\geq 0$.
\end{proof}

We can now prove the following counterpart of Lemma~\ref{lem-3} for unbounded domains:

\begin{lemma}\label{lem-4.3}
Assume $u \in \cal Y$ is a solution to \eqref{P} with Morse index $m(u)\leq 1$.
If\/ $w_\frac \beta2\equiv 0$ in\/~$\Omega_{0\beta}$ then $u$ is constant with respect to\/~$\theta$.
\end{lemma}

\begin{proof}
Since $w_\frac \beta2\equiv 0$ in\/ $\Omega_{0\beta}$, the solution $u$ is symmetric about $\Gamma_{\!\frac\beta2}$, as well as $f'(d(x),z,u)$, while the derivative $u_\theta$ is odd with respect to~$\Gamma_{\!\frac\beta2}$ in the sense that $u_\theta(x) = -u_\theta(\sigma_{\frac \beta 2}(x))$ in~$\Omega_{0\beta}$. We already know that $u_\theta$ satisfies the equation in~\eqref{solution} and vanishes on~$\Gamma_{\! 0} \cup \Gamma_{\! \beta}$. Our goal is to prove that $\partial u_\theta / \partial \nu$ also vanishes there.

\textit{Part I.} We start by showing that $Q_{\Omega_{0\beta}}(\varphi_1) \ge 0$ for every $\varphi_1$ in the vector space $C^{1,\rm odd}_{0}( \Omega_{0\beta} \allowbreak \cup\Gamma_{\! 0} \cup \Gamma_{\! \beta})$
 of the odd elements
 of $C^1_{0}(\Omega_{0\beta}\cup\Gamma_{\! 0} \cup \Gamma_{\! \beta})$. In fact, if $\varphi_1$ vanishes identically the claim is obvious. Otherwise we consider  the absolute value $\varphi_2(x) = |\varphi_1(x)|$, which is symmetric about~$\Gamma_{\!\frac\beta2}$ and therefore
$$
\int_{\Omega_{0\beta}}
f'(d(x),z,u) \, \varphi_1 \, \varphi_2 \, dx
=
0
.
$$
Moreover, we have $\nabla \varphi_2 = \nabla \varphi_1$ in the set $\Omega^+_{0\beta} = \{\, x \in \Omega_{0\beta} : \varphi_1(x) > 0 \,\}$, and  $\nabla \varphi_2 = -\nabla \varphi_1$ in $\Omega^-_{0\beta} = \{\, x \in \Omega_{0\beta} : \varphi_1(x) < 0 \,\}$. Since $\Omega^-_{0\beta} = \sigma_{\frac \beta 2}(\Omega^+_{0\beta})$ by the oddity of~$\varphi_1$, we find
$$
\int_{\Omega_{0\beta}}
\nabla \varphi_1 \, \nabla \varphi_2 \, dx
=
\int_{\Omega^+_{0\beta}}
|\nabla \varphi_1|^2 \, dx
-
\int_{\Omega^-_{0\beta}}
|\nabla \varphi_1|^2 \, dx
=
0
.
$$
Thus, for every $(a,b) \ne (0,0)$ we may compute
\begin{align*}
Q_{\Omega_{0\beta}}(a \, \varphi_1 + b \, \varphi_2)
&=
\int_{\Omega_{0\beta}}
\Big(
|a \, \nabla \varphi_1 + b \, \nabla \varphi_2|^2
- (a \, \varphi_1 + b \, \varphi_2)^2 \, f'(d(x),z,u)
\Big)
\, dx
\\
\noalign{\medskip}
&=
a^2 \, Q_{\Omega_{0\beta}}(\varphi_1)
+
b^2 \, Q_{\Omega_{0\beta}}(\varphi_2)
=
(a^2 + b^2) \, Q_{\Omega_{0\beta}}(\varphi_1)
.
\end{align*}
Since $\varphi_1,\varphi_2$ are linearly independent, and since $m(u) \le 1$ by assumption, we must have $Q_{\Omega_{0\beta}}(\varphi_1) \ge 0$, as claimed.

\textit{Part II.} We claim that $\B(u_\theta,\psi) = 0$ for every $\psi \in C^{1,\rm odd}_{0}( \Omega_{0\beta}\cup\Gamma_{\! 0} \cup \Gamma_{\! \beta})$, which is not obvious because $\psi$ is allowed to take nonzero values on $\Gamma_{\! 0} \cup \Gamma_{\! \beta}$. However, by Part~I the bilinear form $\B$ is positive semidefinite in $C^{1,\rm odd}_{0}( \Omega_{0\beta}\cup\Gamma_{\! 0} \cup \Gamma_{\! \beta}) \times C^{1,\rm odd}_{0}( \Omega_{0\beta}\cup\Gamma_{\! 0} \cup \Gamma_{\! \beta})$ and therefore the Cauchy-Schwarz inequality~\eqref{cs0} holds. Without loss of generality we may assume that the cut-off functions $\xi_R(x)$ are symmetric with respect to~$\Gamma_{\!\frac\beta2}$: thus, $\xi_R(x) \, u_\theta(x)$ is odd and has a compact support. Furthermore, since $f'(r,z,u)$ is locally H\"older continuous and $u \in \cal Y$, $u_\theta$ is in $C^1(\Omega_{0\beta} \cup \Gamma_{\! 0} \cup \Gamma_{\! \beta}) \cap C^0(\overline \Omega_{0\beta}\setminus \Upsilon)$: however, its gradient may fail to belong to $L^2(\Omega_{0\beta} \cap B_{2R})$, due to the possible singularity at $d(x) = r_1$. 
To overcome this difficulty, we use the cut-off functions $\zeta_\delta(x)$ introduced in~\eqref{zeta}. The doubly truncated function $\zeta_\delta \, \xi_R \, u_\theta$ is still odd, has a compact support, and belongs to
$C^{1,\rm odd}_{0}( \Omega_{0\beta}\cup\Gamma_{\! 0} \cup \Gamma_{\! \beta})$. Thus, we may write
$$
0 \le \big( \B(\zeta_\delta \, \xi_R \, u_\theta, \, \psi) \big)^2
\leq
Q_{\Omega_{0\beta}}(\zeta_\delta \, \xi_R \, u_\theta)
\,
Q_{\Omega_{0\beta}}(\psi)
,
$$
where $\psi \in C^{1,\rm odd}_{0}( \Omega_{0\beta}\cup\Gamma_{\! 0} \cup \Gamma_{\! \beta})$. Keeping $\psi$ fixed, we take $\delta$ small and $R$ large to achieve $\zeta_\delta \, \xi_R = 1$ on the set $\mathop{\rm supp} \psi$. Consequently, the preceding inequalities reduce to
$$
0 \le \big( \B(u_\theta, \, \psi) \big)^2
\leq
Q_{\Omega_{0\beta}}(\zeta_\delta \, \xi_R \, u_\theta)
\,
Q_{\Omega_{0\beta}}(\psi)
.
$$
By Proposition~\ref{small}, the term $Q_{\Omega_{0\beta}}(\zeta_\delta \, \xi_R \, u_\theta)$ in the right-hand side becomes arbitrarily small by a convenient choice of $R$ and~$\delta$, which implies $\B(u_\theta, \, \psi) = 0$.

\textit{Part III.} Let us verify that $u_\theta$ satisfies the Neumann condition on $\Gamma_{\! 0}$. 
Recall that $\B(u_\theta, \, \psi) = 0$ by Part II, and take $R > r_1$ such that $\mathop{\rm supp} \psi \subset B_R$. Multiplying the equation $L_u \, u_\theta = 0$ by $\psi\in C^{1,\rm odd}_{0}( \Omega_{0\beta}\cup\Gamma_{\! 0} \cup \Gamma_{\! \beta})$ and integrating by parts over $\Omega_{0\beta} \cap B_R$ we obtain
$$
0=\B(u_\theta,\psi)
=
\int_{\Gamma_{\! 0} \cup \Gamma_{\! \beta}}
\psi \, \frac{\, \partial u_\theta \,}{\partial \nu} \, d\Sigma=2\int_{\Gamma_{\! 0}}
\psi \, \frac{\, \partial u_\theta \,}{\partial \nu} \, d\Sigma
,
$$
where the equality holds because $\psi$ vanishes in a neighborhood of~$\gamma_{0\beta}$ as well as on $\Omega_{0\beta} \cap \partial B_R$ and $\psi \, \frac{\, \partial u_\theta \,}{\partial \nu}$ is symmetric with respect to $\Gamma_{\!\frac\beta2}$. Since $\psi$ can take any value on~$\Gamma_{\! 0}$ then $ \frac{\, \partial u_\theta \,}{\partial \nu}=0$ on $\Gamma_0$ as claimed.

\goodbreak\textit{Conclusion.} The derivative $u_\theta$ can be extended to zero in the domain $\Omega_{-\frac\beta2,0}$ (cf.~\eqref{extension}), thus obtaining a $C^1$ weak solution of $L_u \, u_\theta = 0$ in $\Omega_{-\frac\beta2,\frac\beta2}$. By the unique continuation property, $u_\theta$ must vanish identically, and therefore $u$ is constant with respect to~$\theta$. 
\end{proof}

When the domain $\Omega_{0\beta}$ is unbounded, in place of Lemma~\ref{lem-4} we have:

\begin{lemma}
Assume $u \in \cal Y$ is a solution to~\eqref{P} with Morse index $m(u)\leq 1$. Assume further that \eqref{4.1b} holds. If\/ $w_{\frac \beta 2}(x)>0$ in~$\Omega_{0\frac \beta 2}$ then
\begin{equation}\label{u_theta>0-bis}
u_\theta > 0\ \ \text{in\/ } \Omega_{0\frac \beta 2}\cup \Gamma_{\!\frac \beta 2}.
\end{equation}
\end{lemma}

\begin{proof}
Let us check that the negative part $u_\theta^- = \min\{\, 0,u_\theta \,\} \le 0$ is a weak solution of $L_u \, u_\theta^-=0$ in $\Omega_{0\frac \beta 2}$, i.e., $\B(u_\theta^-,\psi) = 0$ for all $\psi \in  C^\infty_0(\Omega_{0\frac\beta2})$. Assumption~\eqref{4.1b} implies that the bilinear form $\B$ is positive semidefinite on $C^1_{0}(\Omega_{0\frac\beta2}\cup \Gamma_{\! 0}) \times C^1_{0}(\Omega_{0\frac\beta2}\cup\Gamma_{\! 0})$. Hence we pick $\psi \in C^\infty_0(\Omega_{0\frac\beta2})$ and consequently $R > r_1$ such that $\mathop{\rm supp} \psi \subset B_R$, and also $\delta > 0$ such that
\begin{equation}\label{delta}
\inf_{x \in \mathop{\rm supp} \psi} d(x) > r_1 + \delta
.
\end{equation}
As in the proof of Lemma~\ref{lem-4}, the assumption $w_\frac\beta2 > 0$ in~$\Omega_{0\frac\beta2}$ implies $u_\theta > 0$ on~$\Gamma_{\!\frac\beta2}$ (see~\eqref{negative}), hence $u_\theta^- = 0$ on $\gamma_{0\frac\beta2} \cup \Gamma_{\!0} \cup \Gamma_{\!\frac\beta2}$. Furthermore, the doubly truncated function $\zeta_\delta \, \xi_R \, u_\theta^-$ has a square-summable gradient and therefore belongs to $H^1_0(D)$, where $D = \Omega_{0\frac\beta2} \cap B_{2R}$. The Cauchy-Schwarz inequality~\eqref{cs0} yields
$$
0
\le
\big( \B(u_\theta^-, \psi) \big)^2
=
\big( \B(\zeta_\delta \, \xi_R \, u_\theta^-, \, \psi) \big)^2
\le
Q_{\Omega_{0\frac\beta2}}(\zeta_\delta \, \xi_R \, u_\theta^-) \, Q_{\Omega_{0\frac\beta2}}(\psi)
.
$$
By Proposition~\ref{small}, the term $Q_{\Omega_{0\frac\beta2}}(\zeta_\delta \, \xi_R \, u_\theta^-)$ becomes arbitrarily small by a convenient choice of $R$ and~$\delta$, which implies $\B(u_\theta^-,\psi) = 0$ and therefore $u_\theta^-$ is a weak solution of $L_u \, u_\theta^- = 0$ in~$\Omega_{0\frac\beta2}$, as claimed. By the unique continuation property, this implies that either $u_\theta^-\equiv 0$ or $u_\theta^-<0$ in~$\Omega_{0\frac \beta 2}$. But since $u_\theta>0$ on~$\Gamma_{\! \frac\beta2}$, we must have $u_\theta^-\equiv 0$ and therefore $u_\theta \ge 0$ in~$\Omega_{0\frac\beta2}$. The unique continuation property of~$u_\theta$ proves~\eqref{u_theta>0-bis}.
\end{proof}

\section{Proof of Theorem \ref{teo-2}}\label{se:10}

\begin{proof}[Proof of Theorem \ref{teo-2}]
If $w_\frac \beta2 \equiv 0$, then $u$ is independent of~$\theta$ by Lemma~\ref{lem-4.3}. If, instead, $w_\frac \beta2\not \equiv 0$, then Lemma~\ref{lem-4.2} implies that either $w_\frac \beta2>0$ in $\Omega_{0\frac\beta2}$ or $w_\frac \beta2>0$ in $\Omega_{\frac\beta2\beta}$. By Lemma \ref{lem-4.1} we may assume that \eqref{4.1} holds and $w_\frac \beta2>0$ in $\Omega_{0\frac\beta2}$. We then define the set $A'$ as in the proof of Theorem \ref{teo-1} and, arguing as in Part~I with \eqref{u_theta>0-bis} instead of \eqref{u_theta>0}, we get that $A'$ contains~$\frac \beta2$. The fact that $A'$ is a closed subinterval of $\left[\frac \beta 2,\beta\right]$ follows as in Part~II. Next, exactly as in Part~III, the Hopf boundary point lemma implies that $u_\theta>0$ in $\Omega_{0,\tilde\alpha}$ if $A'=\left[\frac\beta 2,\tilde\alpha\right]$, $\tilde \alpha < \beta$. In such a case we can repeat exactly Part~IV getting that
$$
\mbox{$w_{\tilde\alpha}>0$ in $\Omega_{0\tilde\alpha}$.}
$$
If the Morse index $m(u)$ is~$1$, there must be $\varphi_1 \in C^1_{0}(\Omega_{0\beta}\cup\Gamma_{\! 0} \cup \Gamma_{\! \beta})$ such that $Q_{\Omega_{0\beta}}(\varphi_1) < 0$ and we may take a torus~$T$ containing $\mathop{\rm supp} \varphi_1$. By a torus we mean an open subset $T \subset \mathbb R^N$, invariant under cylindrical rotations, whose intersection with the hyperplane $x_1 = 0$ is an $N-1$ dimensional ball lying in the half-plane $x_1 = 0$, $x_2 > r_1 + \varepsilon$ for some $\varepsilon > 0$. The last condition implies that $d(x) \ge r_1 + \varepsilon > r_1$ for all $x \in \overline T$. In the case when $m(u) = 0$ we take any torus~$T$ as defined above. Note that when $N = 2$ the tori are just annuli. Since $w_{\tilde\alpha}$ and $u_\theta$ are positive in~$\Omega_{0\tilde\alpha}$, we obviously have
\begin{equation}\label{obvious}
\mbox{$w_{\tilde\alpha}, u_\theta > 0$ in~$\Omega_{0\tilde\alpha} \cap \overline T$.}
\end{equation}

\textit{Part V.} By the same argument used to derive~\eqref{still} and~\eqref{complete} we see that $w_\alpha > 0$ on~$\Gamma_{\!0}$ for every $\alpha \in (0,\beta)$.

\textit{Part VI.} We claim that $w_\alpha$ keeps positive in $G_\alpha = (\Omega_{0\alpha} \cup {\Gamma_{\!0}}) \cap \overline T$ for every $\alpha > \tilde\alpha$ such that the difference $\alpha - \tilde\alpha$ is sufficiently small. This can be proved by contradiction: suppose that there exists a sequence $\alpha_n \searrow \tilde\alpha$ such that $w_{\alpha_n}$ becomes negative somewhere in $G_{\alpha_n}$, and let $x_n \in \overline \Omega_{0\alpha_n} \cap \overline T$
be a point where $w_{\alpha_n}$ takes its negative minimum. Of course, $x_n \not \in \Gamma_{\!\alpha_n}$ because $w_{\alpha_n} = 0$ there, and $x_n \not \in \Gamma_{\!0}$ by Part~V. By compactness there exists a subsequence, still denoted by $(x_n)$, converging to some $x_0 \in \overline \Omega_{0\alpha} \cap \overline T$ with $w_{\tilde\alpha}(x_0) = 0$. Since $w_{\tilde\alpha} > 0$ on $\Gamma_{\!0}$ by Part~V, and by~\eqref{obvious}, we must have
\begin{equation}\label{location}
\mbox{$x_0 \in \Gamma_{\tilde\alpha} \cap \overline T$.}
\end{equation}
Furthermore, the gradient $\nabla w_{\alpha_n}$ vanishes at~$x_n$ in the case when $x_n \in \Omega_{0\alpha_n} \cap T$, and is directed towards the interior of~$T$ if $x_n \in \Omega_{0\alpha_n} \cap \partial T$. Hence, we may write
$$
\mbox{$\partial w_{\alpha_n} / \partial \theta = 0$ at~$x_n$ for every~$n$.}
$$
Passing to the limit, this property should be inherited by $\partial w_{\tilde\alpha} / \partial \theta$ at~$x_0$. However, by~\eqref{location} and by the Hopf boundary point lemma we have $\partial w_{\tilde\alpha} / \partial \theta < 0$ at~$x_0$, a contradiction. Hence, $w_\alpha$ must be positive in $G_\alpha$ for every $\alpha > \tilde\alpha$ such that the difference $\alpha - \tilde\alpha$ is sufficiently small, as claimed.

\textit{Part VII.} For every $\alpha$ as above, $w_\alpha$ is in fact positive in all of $\Omega_{0\alpha}$. To see this, we let $K = \overline T$ in Lemma~\ref{sign-preservation}, hence $D_\alpha = \Omega_{0\alpha} \setminus \overline T$. The quadratic form $Q_{\Omega_{0\alpha}}$ is positive semidefinite in $C^1_{0}(\Omega_{0\alpha} \cup \Gamma_{\! 0} \cup \Gamma_{\! \alpha} \setminus \overline T)$ by Remark~\ref{stability}, and keeps positive semidefinite in $H^1_{\gamma_{0\alpha} \cup \partial T \cup \partial B_R}(D_\alpha \cap B_R)$ for $R > \max_{x \in \overline T} d(x)$. Taking into account the boundary conditions $w_\alpha > 0$ on $(\Omega_{0\alpha} \cap \partial T) \cup \Gamma_{\!0}$ and $w_\alpha = 0$ on $\gamma_{0\alpha} \cup \Gamma_{\!\alpha}$, by Lemma~\ref{sign-preservation} it follows that $w_\alpha$ is positive in any connected component of $D_\alpha$ and consequently in $\Omega_{0\alpha}$.

\textit{Conclusion.} We have thus seen that if we suppose $\tilde\alpha < \beta$ we obtain the positivity of~$w_\alpha$ in~$\Omega_{0\alpha}$ for every $\alpha > \tilde\alpha$ such that the difference $\alpha - \tilde\alpha$ is sufficiently small, contradicting the definition of~$\tilde\alpha$. Hence, under the assumption that $w_\frac\beta2 > 0$ in~$~\Omega_{0\frac\beta2}$, we must have $\tilde\alpha = \beta$ and therefore the solution $u$ is strictly increasing with respect to~$\theta$ in~$\Omega_{0\beta}$. In this case, by~Part~III we also have~$u_\theta > 0$ in~$\Omega_{0\beta}$ and we can apply Lemma 2.1 in \cite{gladiali-pacella-weth} in $\Omega_{0\beta}$ getting that $\inf Q_{\Omega_{0\beta}} (\psi)\geq 0$ when $\psi\in C^1_0(\Omega_{0\beta})$.
By Proposition~\ref{small}, the term $Q_{\Omega_{0\beta}}(\zeta_\delta \, \xi_R \, u_\theta)$ becomes arbitrarily small by a convenient choice of $R$ and~$\delta$ showing that $\inf_{\psi \in C^1_0(\Omega_{0\beta})} Q_{\Omega_{0\beta}} (\psi)= 0$.
 In summary,  we have shown that if $m(u) \le 1$ then either $u_\theta$ vanishes identically in~$\Omega_{0\beta}$ or it keeps its sign there. Now, in order to prove Claim~1 of the statement, it suffices to check that if $m(u) = 0$ then $u_\theta$ vanishes identically. To this purpose, take $\psi \in C^1_{0}( \Omega_{0\beta}\cup\Gamma_{\! 0} \cup \Gamma_{\! \beta})$ and consider the doubly truncated function $\zeta_\delta \, \xi_R \, u_\theta$ with $R$ such that $\mathop{\rm supp} \psi \subset B_R$ and $\delta$ satisfying~\eqref{delta}. Since $Q_{\Omega_{0\beta}}$ is positive semidefinite by assumption, and since $\zeta_\delta \, \xi_R \, u_\theta$ belongs to $H^1_{\gamma_{0\beta} \cup \partial B_{2R}}(D)$, where $D = \Omega_{0\beta} \cap B_{2R}$, the Cauchy-Schwarz inequality yields
$$
0
\le
\big( \B(u_\theta, \, \psi) \big)^2
=
\big( \B(\zeta_\delta \, \xi_R \, u_\theta, \, \psi) \big)^2
\le
Q_{\Omega_{0\beta}}(\zeta_\delta \, \xi_R \, u_\theta) \, Q_{\Omega_{0\beta}}(\psi)
.
$$
By Proposition~\ref{small}, the term $Q_{\Omega_{0\beta}}(\zeta_\delta \, \xi_R \, u_\theta)$ becomes arbitrarily small by a convenient choice of $R$ and~$\delta$. On the other side, multiplying the equation $L_u \, u_\theta = 0$ by $\psi$ and integrating by parts we obtain
$$
{\cal B}_{\Omega_{0\beta}}(u_\theta,\psi)
=
\int_{\Gamma_{\!0} \cup \Gamma_{\!\beta}}
\psi \, \frac{\, \partial u_\theta \,}{\, \partial \nu \,}
\, d\Sigma
,
$$
hence the integral must vanish. Since $\psi$ can take any value on $\Gamma_{\!0} \cup \Gamma_{\!\beta}$, it follows that $\partial u_\theta / \partial \nu = 0$ there, whence $u_\theta$ can be extended by zero to $\Omega_{-\beta,0}$ as in~\eqref{extension}, thus obtaining a $C^1$ weak solution $\tilde u_\theta$ of $L_u \, \tilde u_\theta = 0$ in $\Omega_{-\beta,\beta}$. Since $\tilde u_\theta$ vanishes identically in $\Omega_{-\beta,0}$, by the unique continuation property we must have $u_\theta = 0$ in $\Omega_{0\beta}$, hence $u$ is independent of~$\theta$, and the proof is complete.
\end{proof}

\pdfbookmark{Acknowledgements}{Acknowledgements}
\section*{Acknowledgements} The authors are members of the Gruppo
Nazionale per l'A\-na\-li\-si Matematica, la Pro\-ba\-bi\-li\-t\`a e
le loro Applicazioni
(\href{https://www.altamatematica.it/gnampa/}{GNAMPA}) of the Istituto
Na\-zio\-na\-le di Alta Ma\-te\-ma\-ti\-ca
(\href{https://www.altamatematica.it/en/}{INdAM}). F.~Gladiali is supported by Uniss, \href{https://www.fondazionedisardegna.it/}{Fondazione di Sardegna}, annualità 2017. A.~Greco is
partially supported by the research project {\em Evolutive and stationary Partial Differential Equations with a focus on bio-mathematics}, funded by
\href{https://www.fondazionedisardegna.it/}{Fondazione di Sardegna}, annualità 2019. This paper is dedicated to the city of Alghero (Sardinia, Italy) where the authors repeatedly met to carry their work on.

\bigskip

\noindent\begin{tabular}{ p{78mm}l }
Francesca Gladiali                  & Antonio Greco\\
Dipartimento di Chimica e Farmacia  & Dipartimento di Matematica e Informatica\\
Universit\`a degli Studi di Sassari & Universit\`a degli Studi di Cagliari\\
Italy                               & Italy\\
 e-mail: fgladiali@uniss.it         & e-mail: greco@unica.it\\
\end{tabular}

\end{document}